\documentclass[10pt, oneside, a4paper]{article}
\usepackage{geometry}
\title{Isogenies of elliptic curves over function fields}
\author{{\Large Richard {\sc Griffon}} \and {\Large Fabien {\sc Pazuki}}}	
\date{}
\usepackage[english]{babel}
\usepackage[utf8]{inputenc}\usepackage{lmodern}\usepackage{amssymb,amsmath,amscd, amsthm}\usepackage{mathrsfs}\usepackage{tikz-cd}\usepackage{enumerate}\usepackage{graphicx}
\usepackage{titlesec}
\titleformat{\subsection}[runin]{\bfseries}{\thesubsection.}{0.2em}{}[.\hspace{0.4em}-- ]        
\titleformat{\section}{\center\Large\bfseries}{\thesection.}{0.2em}{}[]
\usepackage{url}
\definecolor{BrickRed}{RGB}{153, 0, 0}
\usepackage[colorlinks=true, citecolor= orange, linkcolor= BrickRed!80, urlcolor=black!60, pagebackref]{hyperref}

	\newtheorem{theo}{Theorem}[section]
	\newtheorem{coro}[theo]{Corollary}
	\newtheorem{lemm}[theo]{Lemma}
	\newtheorem{prop}[theo]{Proposition}
	{\theoremstyle{definition}
	\newtheorem{defi}[theo]{Definition}
	
	\newtheorem{rema}[theo]{Remark}}	
	{\newtheorem{itheo}{Theorem}}
	
	{\newtheorem{icoro}[itheo]{Corollary}}

			\DeclareMathOperator{\Gal}{Gal}
			\DeclareMathOperator{\Hom}{Hom}			
			\DeclareMathOperator{\Div}{Div}
			\DeclareMathOperator{\DIVIS}{div}
			\renewcommand{\div}{\DIVIS}			
      		
      		\DeclareMathOperator{\Spec}{Spec}
			\DeclareMathOperator{\DEGRE}{deg }
			\renewcommand{\deg}{\DEGRE}
			
			\newcommand{\R}{\ensuremath{\mathbb{R}}}
			\newcommand{\Q}{\ensuremath{\mathbb{Q}}}
			\newcommand{\C}{\ensuremath{\mathbb{C}}}
			\newcommand{\Z}{\ensuremath{\mathbb{Z}}}
			\newcommand{\G}{\ensuremath{\mathbb{G}}}
			\renewcommand{\P}{\ensuremath{\mathbb{P}}}
			
			\newcommand{\Ecal}{\mathcal{E}}
			\newcommand{\Acal}{\mathcal{A}}			
			\newcommand{\e}{\ensuremath{\mathrm{e}}}

			\renewcommand{\bar}[1]{\ensuremath{\overline{#1}}}
			\renewcommand{\hat}[1]{\ensuremath{\widehat{#1}}}
			\newcommand{\cond}{\mathcal{N}}
			
			\newcommand{\sep}{^\mathrm{sep}}
			\newcommand{\ie}{\textit{i.e.}{}}

			\newcommand{\Frob}[1]{\mathrm{Fr}_{#1}}
			\newcommand{\Vers}[1]{\mathrm{V}_{\!#1}}
			\newcommand{\SL}{\mathrm{SL}}
            \newcommand{\hmod}{\mathrm{h}_{\mathsf{mod}}}
            \newcommand{\hstab}{\mathrm{h}_{\mathsf{st}}}
            \newcommand{\hdiff}{\mathrm{h}_{\mathsf{diff}}}
            \newcommand{\degins}{\mathrm{deg}_{\mathrm{ins}}}
            \newcommand{\degsep}{\mathrm{deg}_{\mathrm{sep}}}
			\newcommand{\Deltamin}{\Delta_{\min}}
			\renewcommand{\O}{\mathcal{O}}
			\newcommand{\Escr}{\mathscr{E}}
			\newcommand{\Iscr}{\mathscr{I}}
			\newcommand{\Jscr}{\mathcal{J}}			
			\newcommand{\E}{\mathbb{E}}
			\renewcommand{\hat}{\widehat}			  
			\newcommand{\Ycal}{\mathcal{Y}}
			\newcommand{\Xcal}{\mathcal{X}}
			\newcommand{\Hcal}{\mathcal{H}}

			\newcommand{\cO}{49}
			\newcommand{\cycsbgp}{\mathscr{C}}

\begin{document}
\pagestyle{plain}
\maketitle 

\noindent\hfill\rule{7cm}{0.5pt}\hfill\phantom{.}

\paragraph{Abstract -- } 
We prove two theorems concerning isogenies of elliptic curves over function fields. 
The first one describes the variation of the height of the $j$-invariant in an isogeny class. 
The second one is an ``isogeny estimate'', providing an explicit bound on the degree of a minimal isogeny between two isogenous elliptic curves. 
We also give several corollaries of these two results. 

 
\smallskip
\noindent{\it 2020 Math. Subj. Classification:} 
11G05, 
11G35, 
11G50, 
11R58, 
14G17, 
14G25, 
14G40, 
14H52, 
14K02. 

\noindent\hfill\rule{7cm}{0.5pt}\hfill\phantom{.}

\section*{Introduction}
\pdfbookmark[0]{Introduction}{Introduction}
\addcontentsline{toc}{section}{Introduction}\setcounter{section}{0}
 
Let $k$ be a perfect field, and $C$ be a smooth projective and geometrically irreducible curve over $k$. 
We write $K:=k(C)$ for the function field of $C$, and fix an algebraic closure $\bar{K}$ of $K$.

It is well-known that two elliptic curves defined over $\bar{K}$ are isomorphic 
if and only if they have the same $j$-invariant. 
One may wonder, more generally, about arithmetic relations between 
the $j$-invariants of two elliptic curves which are linked by an isogeny of arbitrary degree.

Denoting the Weil height on $\bar{K}$ by  ${h}(.)$, we prove the following.
		\begin{itheo}\label{itheo:comp.hmod}
		Let $E_1$, $E_2$ be two non-isotrivial elliptic curves defined over $\bar{K}$ 
		with respective $j$-invariants $j(E_1)$, $j(E_2)$. 
		Assume that there exists an isogeny $\varphi:E_1\to E_2$, and denote its dual by $\hat{\varphi}:E_2\to E_1$. 

		We have
		\begin{equation}\label{6}
		 {h}(j(E_{2})) = \frac{\degins(\varphi)}{\degins(\hat{\varphi})} \, {h}(j(E_{1})).
		\end{equation}
		Here $\degins(\varphi)$ and $\degins(\hat{\varphi})$ denote the inseparability degrees 
		of $\varphi$ and $\hat{\varphi}$ (see \S\ref{ss:isogenies.prelim} for the  definition);
		if~$K$ has characteristic $0$, these should be interpreted as $1$.
		\end{itheo}

In the particular case where the field $K$ has characteristic $0$, the above result states that 
\emph{isogenies preserve the height of the $j$-invariant}.

Theorem \ref{itheo:comp.hmod} should be compared to its analogue for elliptic curves over number fields: 
inspired by classical work of Faltings and Silverman,
Theorem~1.1 in \cite{Paz18} asserts that 
the $j$-invariants of two elliptic curves $E_1, E_2$ defined over $\bar{\Q}$  which are linked by an isogeny $\varphi$ satisfy: 
\begin{equation}\label{eq1}
\left|ht(j(E_1)) - ht(j(E_2))\right| \leq 9.204 + 12 \log\deg\varphi,
\end{equation}
where $ht(.)$ is the logarithmic Weil height on $\bar{\Q}$. 
Results of Szpiro--Ullmo \cite{SzpiroUllmo} imply that the displayed upper bound is almost optimal -- 
in the sense that the Weil heights of $j$-invariants of elliptic curves related by an isogeny can actually differ 
by a quantity of the same order of magnitude as a multiple of $\log\deg\varphi$. 
This latter statement in turn implies that, given an elliptic curve $E/\bar{\Q}$ without CM, the set
\[ \Jscr(E, B) 
:= \left\{ j(E')\in\bar{\Q} : E'/\bar{\Q} \text{ is isogenous to $E$, and }ht(j(E'))\leq B\right\}\]
is finite for all $B\geq 1$.
Theorem~\ref{itheo:comp.hmod} provides a much tighter control on the variation of the Weil height of 
the $j$-invariant in isogeny classes of elliptic curves over function fields than \eqref{eq1} does  for elliptic curves over number fields. 
Which leads to the surprising consequence that the sets which are natural analogues 
for~$\Jscr(E,B)$ are infinite in the function field setting (see Proposition~\ref{prop:infinite})! 

After introducing the relevant objects in sections \ref{sec:intro.FF}--\ref{sec:heights.EC} and
proving some preliminary results about isogenies in section \ref{sec:isogenies},
we show Theorem \ref{itheo:comp.hmod} in section \ref{sec:isog.heights}. 
Our proof relies on a comparison of modular and differential heights of elliptic curves, 
and on a result of Parshin stating that the differential height is invariant under isogeny of degree coprime to the characteristic.
As far as we know (see Remark \ref{rema:alternative.approach}), this is the most direct approach. 
 
We refer to the recent \cite{BPR21} for an analogue of Theorem \ref{itheo:comp.hmod} for isogenous Drinfeld modules. 

\paragraph{}
Let us now turn to the second main theorem of this article, 
which is an isogeny estimate for elliptic curves over function fields.
The result (Theorem \ref{theo.small.isogenies}) may be stated as follows:

		\begin{itheo}\label{itheo:min.isog}	
		Let $E_1$ and $E_2$ be two non-isotrivial elliptic curves defined over a function field $K$ of genus~$g$, 
		with respective $j$-invariants $j(E_1)$, $j(E_2)$. 
		Assume that $E_1$ and $E_2$ are isogenous. 
		Then there exists a $\bar{K}$-isogeny $\varphi_0:E_1\to E_2$ with
		\[ \deg\varphi_0 
		\leq 49 \max\{1, g\}\, \max\left\{\frac{\degins j(E_1)}{\degins j(E_2)},\frac{\degins j(E_2)}{\degins j(E_1)}\right\},\]
		where $\degins j(E_1), \degins j(E_2) $ are the inseparability degrees of 
		$j(E_1), j(E_2)$ respectively (see \S\ref{ssec:Insepdeg} for the definition); 
		if~$K$ has characteristic $0$, they should be interpreted as $1$.
		\end{itheo}

Isogeny estimates have attracted a lot of attention in the number field case since 
the groundbreaking series of papers of Masser and W\"ustholtz \cite{MasserW1,MasserW2}:
Masser's and W\"ustholz's results were later refined by Pellarin \cite{Pellarin}, 
and more recently by Gaudron and R\'emond \cite{GaudronRemond}.
Each of these papers provides the existence of a ``small'' isogeny between 
a pair of isogenous elliptic curves (in fact \cite{MasserW2, GaudronRemond} even deal with general Abelian varieties). 
Here ``small'' means that the degree of the isogeny is bounded (at worst) in terms of 
the height of the elliptic curves and  simple invariants of the base field.  
Theorem \ref{itheo:min.isog} has the same flavour as these statements.

The proofs in \cite{MasserW1,MasserW2, Pellarin, GaudronRemond}  heavily rely on transcendence methods, 
and uniformisation at an Archimedean place. 
In the context of function fields, all places are non-Archimedean and, 
as far as the authors know, only few transcendence results are available in positive characteristic.
Our proof therefore follows a different strategy as the above mentioned works: 
we essentially rely  on properties of modular curves~$X_0(N)$, the Riemann--Hurwitz formula and an estimate of the genus of $X_0(N)$. 
Note that one could also follow a similar approach relying instead on the gonality of $X_0(N)$ in place of its genus. 
We also note that David and Denis have proved an analogous statement 
for isogenies between Drinfeld modules (see \cite[Th\'eor\`eme 1.3]{DavidDenis}).
 
In Theorem \ref{itheo:min.isog} too, the case where $K$ has characteristic $0$ offers the most striking result: 
in this situation, \emph{the degree of $\varphi_0$ can be bounded independently of $E_1, E_2$}. 
In other words, Theorem \ref{itheo:min.isog} provides a \emph{uniform} isogeny estimate for 
elliptic curves over a function field of characteristic $0$.
In positive characteristic~$p$, one can easily see that such an isogeny estimate cannot be uniform: 
for any $n\geq 1$, the smallest isogeny between a non-isotrivial elliptic curve $E$ and its $p^n$-th Frobenius twist $E^{(p^n)}$ 
is indeed the $p^n$-th power Frobenius isogeny $\Frob{p^n}:E\to E^{(p^n)}$, which has degree $p^n$. 
In that situation, the dependency of Theorem~\ref{itheo:min.isog} on the curves $E_1, E_2$ 
(through the inseparability degree of their $j$-invariants) is actually best possible.

Theorem~\ref{itheo:min.isog} is proved in Section~\ref{sec:isog.estimate} of the paper, 
where we also establish a few corollaries of this isogeny estimate. 
We give, in particular, the following explicit bound on the number of $K$-isomorphism classes 
of elliptic curves in a given $K$-isogeny class (see Theorem \ref{prop:effective.shaf}).

		\begin{icoro}\label{icoro:eff.sha}
		Let $E$ be a non-isotrivial elliptic curve over a function field $K$ of genus $g$.  
		For any $M\geq 1$, consider the set  
		\[\Escr_K(E, M) := \left\{ E'/K \text{ is $K$-isogenous to $E$, with }\degins j(E')\leq M\right\}\big/K\text{-isomorphism}.\]
		The set $\Escr_K(E, M) $ is finite, and we have
		\begin{itemize}
		\item $\displaystyle \big|\Escr_K(E,M)\big| \leq 7^4 \max\{1, g\}^2$ \ 
		if $K$ has characteristic $0$,
		\item $\displaystyle \big|\Escr_K(E,M)\big| \leq 7^4 \max\{1, g\}^2\left(\frac{\log M}{\log p} +1\right)$ \  
		if $K$ has characteristic $p>0$.
		\end{itemize}
		\end{icoro}

\numberwithin{equation}{section}    
\section{Preliminaries about function fields}\label{sec:intro.FF}

Let $k$ be a perfect field of characteristic $p\geq 0$, 
and let $C$ be a smooth projective and geometrically irreducible curve over $k$. 
We let $K:=k(C)$ denote the function field of $C$ over $k$. 
The field $k$ is then algebraically closed in $K$, and we call it the \emph{constant field of $K$}.
It is well-known that the (isomorphism class of the) field $K$ characterises $C$ 
up to birational equivalence over $\bar{k}$ (see \cite{AEC} Chapter II, Remark 2.5 for instance). 
Any smooth projective curve over $k$ whose function field is $K$ will be called a \emph{model} for~$K$. 
The genus of a model $C$ for $K$ will be called the \emph{genus of $K$}.

Any finite extension $L$ of $K$ is also a function field in the above sense. 
Precisely, there is a finite extension $k'/k$ 
and a smooth projective and geometrically irreducible curve $C'/k'$ such that $L=k'(C')$.
The inclusion $K\subset L$ induces a morphism $C'\to C$ between the underlying curves.

We refer to \cite{Rosen} for more details about function fields, in particular to \cite[Chapter V]{Rosen}.

\subsection{Absolute values on $K$} 

For a function field $K$ as above, we let $M_K$ denote the set of \emph{places} of~$K$ \ie{}, 
equivalence classes of discrete valuations on $K$.
Once a model $C$ of $K$ has been chosen, there is a bijection between $M_K$ and the set of closed points on~$C$. 
Given a place $v\in M_K$, the residue field $k_v$ of~$K$ at $v$ is a finite extension of $k$: 
the degree $\deg v:=[k_v:k]$ of this extension will be called the degree of $v$.

To each place $v\in M_K$, one associates an absolute value $|.|_v$ on $K$, defined by 
$|x|_v := \e^{- v(x) \deg v}$ for all $x\in K$,
where $v(x)$ is  the order of $x$ at $v$ (so normalised that $v(K^\ast)=\Z$).
It is classical that $K$ then satisfies a \emph{product formula}: 
\[\forall x \in K^\ast, \qquad \prod_{v\in M_K} |x|_v=1, \quad \text{i.e.,} \ \sum_{v\in M_K} v(x)  \deg v=0;\]
 (see \cite[Chapter I, \S1]{LangFunDio}).
In terms of a model $C$ for $K$, this identity is a reformulation of the fact that 
a rational function on $C$ has as many poles as zeros (
with multiplicities),  see \cite[Proposition 5.1]{Rosen}.

\subsection{Absolute values on finite extensions of $K$}

Let $K$ be as in the previous subsection, and let~$L/K$ be a finite field extension. 
The constant field $k'$ of $L$ is then a finite extension of $k$. 
We let $M_L$ denote the set of places of $L$. 
Given a place $w\in M_L$, we write $v$ for the place of $K$ lying under $w$ 
(\ie{}, $v$ is the restriction of $w$ to $K\subset L$).
The residue field $k'_w$  of $L$ at $w$ is then a finite extension of $k'$, itself a finite extension of $k$; 
and we let $\deg w := [k'_w: k]$.
(This choice might not be the most common one, but  will avoid some notational complications later on).

We associate to $w\in M_L$ a normalised absolute value $|.|_w$ on $L$: the normalisation is the one such that 
$|x|_w=|x|_v$ for all $x\in K$.
The image of $L^\ast$ under $|.|_w$ is then a subgroup of $\R_{>0}$  which contains 
as a subgroup of finite index the image of $K^\ast$ under $|.|_v$.
That index is the \emph{ramification index of $w$}, we denote it by $e_w:=\big( |L^\ast|_w : |K^\ast|_v \big)$.
We denote by $f_w := [k'_w : k_v]$ the \emph{residual degree at $w$}. 
With our normalisations, notice that $\deg w = [k'_w : k] = [k'_w : k_v] \, [k_v : k] = f_w \, \deg v $.
Finally, write $L_w$ and $K_v$ for the completions of~$L$ at $w$ and of $K$ and $v$, respectively, 
and let $n_w:=[L_w:K_v]$ denote the \emph{local degree}.
By \cite[Chapter I, Proposition 2.4]{LangFunDio}, we have $n_w= e_w \, f_w$.
With these definitions at hand, one shows that 
\[\forall x\in L^\ast, \qquad n_w \log |x|_w = -w(x) \deg w.\]

Thus endowed with these absolute values, one can prove that $L$ too satisfies a product formula:
\[\forall x \in L^\ast, \qquad 
\prod_{w\in M_L} |x|_w^{n_w}=1, \quad \text{i.e.,} \ \sum_{w\in M_L} w(x) \deg w=0.\]
The key part of the proof is to show that, for any $x\in L$ and any $v\in M_K$, we have 
$|\mathbf{N}_{L/K}(x)|_v = \prod_{w\mid v} |x|_w$.
The latter identity is usually proved under the assumption that $L/K$ be separable but, 
as noted in \cite{LangFunDio}, it remains true without this assumption, 
provided that $v$ is ``well-behaved'' in the terminology of Lang. 
A general proof may be found in \cite[Chapter I, Theorem 5.3]{DwoGerSul}, 
or deduced from \cite[Theorem 7.6]{Rosen}.

\subsection{Weil height on $\overline{K}$}\label{ssec:hweil}

We use the notation introduced above and we also fix an algebraic closure $\bar{K}$ of $K$. 
Algebraic extensions of $K$ will be viewed as sub-extensions of $\bar{K}$.
Let $P=[x_0:x_1]\in{\mathbb{P}^1(\bar{K})}$ be a point, 
and pick a finite extension $L/K$ over which $P$ is rational. 
We define the relative height of $P$ by 
\begin{equation*}
h_L(P)=\sum_{w\in{M_L}}n_w \log\max\{|x_0|_w, |x_1|_w\}, 
\end{equation*}
and the \emph{absolute logarithmic Weil height} of $P$ by the formula:
\begin{equation}
h(P) := \frac{h_L(P)}{[L:K]} =\frac{1}{[L:K]}\sum_{w\in{M_L}}n_w \log\max\{|x_0|_w,|x_1|_w\}.
\end{equation}
One may check that this last definition does not depend on the choice of an extension $L$ containing $P$, 
nor on the choice of homogeneous coordinates for $P$  (see \cite[Chapter III, \S1]{LangFunDio} or \cite[Proposition 7.7]{Rosen}).
This construction thus defines a height function on $\mathbb{P}^1(\bar{K})$, which takes values in $\Q_{\geq0}$. 
\\

For any $f\in\bar{K}$, we  write $h(f)$ for the absolute logarithmic Weil height 
of the point $[f:1]\in\mathbb{P}^1(\bar{K})$. Explicitly, for any $f\in\bar{K}$, we have
\[h(f) = \frac{1}{[L:K]}\sum_{w\in M_L} n_w \max\{0,\log |f|_w\}
=\frac{1}{[L:K]}\sum_{w\in M_L} \deg w \, \max\{0,-w(f)\},\]
where $L$ is any finite extension of $K$ containing $f$.
Choosing a model $C'/k'$ for $L$, one may view an element $f\in L^\ast$ as a morphism $f:C'\to\P^1$ defined over $k'$. 
Write $\div_\infty(f)\in \mathrm{Div}(C')$ for the divisor of poles of that function.
By the right-most expression of the previous display, we have 
\[ h(f) = \frac{\deg\big(\div_\infty(f)\big)}{[L:K]},\]
which means that $h_L([f:1])=[L:K]h(f)$ equals the degree of $f$, viewed as a morphism $C'\to\P^1$.
\\

For a given $f\in\bar{K}$, note  that $h(f)=0$ if and only if $f$, viewed as a morphism $C'\to\P^1$, is constant 
(\ie{}, $f$ belongs to an algebraic extension of the constant field $k$). 
Indeed, a non-constant morphism $C'\to\P^1$ is surjective so that, in particular, 
its divisor of poles is a non-zero effective divisor. 
Hence the height of this morphism must be positive.

Note that the height $h:\bar{K}\to\Q_{\geq0}$ does not necessarily satisfy the Northcott property:  
unless the constant field $k$ is finite,  
the set $\{f\in K : h(f)=0\}$ (which equals $k$) is not finite.

\subsection{Inseparability degree}\label{ssec:Insepdeg}

Let $K$ be a function field with constant field $k$.
We assume in this subsection that $K$ has \emph{positive} characteristic $p$.
Recall that the \emph{inseparability degree} of an element $f\in K^\ast$ is defined by 
\[\degins(f):=\begin{cases}
1 & \text{ if } f\in k, \\
\big[K:{k}(f)\big]_i & \text{ if } f\notin k,
\end{cases}\] 
where $\big[K:{k}(f)\big]_i$ denotes the inseparability degree of the extension $K/k(f)$ 
(which is finite under the assumption that $f$ be non-constant).
The inseparability degree of $f$ is a non-negative power of $p$.
If $f$ is non-constant, one can prove that $\degins(f)=p^e$ where 
$e\geq 0$ is the maximal integer such that $f\in K^{p^e}$.

Fixing a model $C$ of $K$, we  view $f$ as a morphism $f:C\to\P^1$.
We may then factor $f$ as the composition $f_s\circ F_q$ of a (purely inseparable) Frobenius map $F_q:C\to C^{(q)}$ 
for some power $q$ of $p$, with a separable map $f_s:C^{(q)}\to\P^1$. 
The inseparability degree of $f$ then equals $q = \deg F_q$. 
We refer to \cite[Chapter II, Corollary 2.12]{AEC} for a proof.

For completeness, we also note the following easily proved fact. 
Given a finite extension $K'/K$ with inseparability degree $[K':K]_i$, 
the inseparability degree of an element $f\in K$ viewed as an element of $K'$ is equal to 
\[\degins(f\in K') = \degins(f\in K) \, [K':K]_i.\]

\section{Preliminaries on reduction of elliptic curves}\label{sec:intro.EC}

Let $K$ be a function field (in the sense of section \ref{sec:intro.FF}) with constant field $k$. 
An elliptic curve over $K$ is called \emph{isotrivial} if its $j$-invariant is constant (\ie{} is an element of $k$).
After a finite extension of $K$, an isotrivial elliptic curve becomes isomorphic 
to (the base change of) an elliptic curve defined over a finite extension of $k$.
We will be mostly interested in elliptic curves which are \emph{not} isotrivial.

For more complete overviews of the arithmetic of elliptic curves over function fields, 
the reader is referred to \cite{UlmerParkCity}  and \cite[Chapter III]{ATAEC}.

\subsection{Minimal discriminant and conductor}\label{ssec:mindisc.cond}

Let $E$ be an elliptic curve over $K$.  Let $v\in M_K$ be a place of $K$.
After a suitable change of coordinates, the curve $E$ admits Weierstrass models
\begin{equation}\label{Wmodel}
y^2+a_1 xy+a_3 y=x^3+a_2x^2+a_4x+a_6,
\end{equation}
where the coefficients $a_1, a_2, a_3, a_4, a_6\in K$ are integral at $v$ \ie{}, models 
with $v(a_i)\geq 0$ for $i\in\{1,2,3,4,6\}$.
Each of these models has a discriminant $\Delta(a_1, a_2, a_3, a_4, a_6)\in K$  which, 
being a polynomial in the $a_i$'s, is also integral at $v$.
We write $\delta_v\in\Z_{\geq 0}$ for the minimal value of $v(\Delta(a_1, a_2, a_3, a_4, a_6))$ 
among all models of~$E$ which are integral at $v$.
A $v$-integral Weierstrass model \eqref{Wmodel} of $E$ is called \emph{minimal integral at $v$} 
if the valuation at $v$ of its discriminant equals $\delta_v$.
\\

Given this collection of local data, we define a global invariant of $E$: the \emph{minimal discriminant} of $E/K$ 
is the divisor on $K$ defined by $\Deltamin(E/K)=\sum_{v\in{M_K}}\delta_v \,  v.$ 

We also recall that the \emph{conductor} of $E/K$ is the divisor given by
$\cond({E/K})=\sum_{v\in{M_K}}f_v \,  v$, 
where $f_v\in\Z_{\geq 0}$ is the \emph{local conductor of $E$ at $v$}.
We refer the reader to \cite[Chapter IV, \S10]{ATAEC} for a detailed definition of~$f_v$.

\subsection{Good and semi-stable places}\label{ssec:semistab.red}

Let $E/K$ be an elliptic curve and $v\in M_K$ be a place of $K$. 
We say that $E$ has \emph{good reduction at $v$} if and only if the reduction modulo $v$ 
of one/any integral minimal model of $E$ at $v$ is a smooth curve over the residue field $k_v$ of $K$ at $v$.

More generally, we say that $E$ has \emph{semi-stable reduction at $v$} if 
there exists an integral minimal model of $E$ at $v$ whose reduction modulo $v$ has at most one double point.
This is equivalent to requiring that the reduction of $E$ at $v$ is either good or multiplicative.
The theorem of Kodaira-N\'eron (see \cite[Chapter~VII, \S6, Theorem 6.1]{AEC} for instance)
implies that for any elliptic curve $E/K$ and  any  place $v\in M_K$ of semi-stable reduction for $E$, 
one has $v(\Delta_v)=-v(j)$, where $j$ denotes the $j$-invariant of $E$, 
and $\Delta_v$ the discriminant of a minimal integral model of $E$ at $v$. 

An elliptic curve over $K$ is called \emph{semi-stable} if it 
has semi-stable reduction at every place of $K$. 
For any elliptic curve $E$ over $K$, there exists a finite extension $K'/K$ such that $E\times_KK'/K'$ is semi-stable. 
This statement is the famous \emph{semi-stable reduction theorem} (see \cite[Chapter XI]{MB}).

\subsection{Tate's uniformisation of elliptic curves with non-integral $j$-invariants}\label{ssec:tate.unif} 

In this subsection, we work in the following setting.
Consider a field $F$ which is complete for a non-trivial non-Archimedean valuation $|.|$. 
We review some aspects of Tate's uniformisation of elliptic curves over $F$: 
the reader is referred to Tate's beautiful survey \cite{Tate93} for a more in-depth presentation, 
or to \cite[Chapter V, \S3-\S5]{ATAEC} for an overview in the case where $F$ is a finite extension of~$\Q_p$. 

Let $t\in{F}$ be such that $|t|<1$. 
Consider the  curve defined over $F$ by 
\[\E_t:\qquad Y^2 +XY =X^3 + c_4(t) X + c_6(t),\]
where $c_4, c_6 : \bar{F}\to\bar{F}$ are certain power series which converge on the disc $\{z\in \bar{F}: |z|<1\}$. 
Tate proved that $\E_t$ is  an elliptic curve, whose $j$-invariant $j(\E_t)\in F$ is given 
by a convergent power series in~$t$ and satisfies  $|j(\E_t)| = |t|^{-1}>1$. 
Furthermore, he has shown that there is an  analytic group isomorphism  $\E_t \to \G_m/t^\Z$ (the uniformisation map) 
which is Galois equivariant. In other words, for any algebraic extension $F'$ of $F$, 
there is a group isomorphism $\E_t(F')\simeq (F')^\ast/t^\Z$.
Note that  $\E_t$ has split multiplicative reduction at the maximal ideal of $|.|$. 

Conversely, let $E$ be an elliptic curve over $F$ whose $j$-invariant satisfies $|j(E)|>1$.
Then, there exists a unique $t\in F$ with $|t|<1$ such that $E$ is isomorphic to $\E_t$, 
the isomorphism being defined over an at most quadratic extension  of $F$. See \cite[Chapter V, Theorem 5.3]{ATAEC}.

For $t_1, t_2\in F$ such that $|t_1|<1$ and $|t_2|<1$, let $\Hom(\E_{t_1}, \E_{t_2})$ denote 
the $\Z$-module consisting of isogenies $\E_{t_1}\to \E_{t_2}$ (defined over $\bar{F}$) 
together with the constant morphism equal to $0_{\E_{t_2}}$.
The theorem page 16 of \cite{Tate93} states that $\Hom(\E_{t_1}, \E_{t_2})$ is in bijection 
with the set $\big\{(n_1, n_2)\in\Z^2 : t_1^{n_1} = t_2^{n_2}\big\}$.
In particular, for any $t\in F$ with $|t|<1$, the ring of $\bar{F}$-endomorphisms of $\E_t$ is isomorphic to $\Z$. 
Indeed, the above implies that we have 
\begin{equation}\label{eq:endom.tate.unif}
\Hom(\E_{t}, \E_{t}) \simeq \big\{(n,n), \ n\in\Z\big\},\end{equation}
where the correspondence associates the multiplication-by-$n$ map $\E_t\to\E_t$ to the pair $(n,n)\in\Z^2$.

\subsection{CM and isotriviality}\label{ssec:CM.isotriv} 

Over a number field, it is well known that an elliptic curve with non-integral $j$-invariant 
does not have CM (see \cite[Chapter 5, \S6, Theorem 6.3]{ATAEC}). 

The analogue statement also holds over function fields:
		\begin{lemm}\label{lemm:CM.isotriv}
		Let $K$ be a function field as above, and let $E$ be an elliptic curve over $K$ whose $j$-invariant is not constant. 
		Then, the  ring $\mathrm{End(E)}$ of $\bar{K}$-endomorphisms of $E$  is isomorphic to $\Z$.
		In other words, the curve $E$ has ``no complex multiplication''.
		\end{lemm}

By the work of Serre--Tate \cite{SerreTate}, we know that an elliptic curve defined over a local field and 
with complex multiplication has integral $j$-invariant, hence has potentially good reduction. 
From this we infer that an elliptic curve defined over a function field $K$, with complex multiplication 
and with potentially good reduction everywhere must have integral $j$-invariant at all places of $K$. 
In particular its $j$-invariant, having no poles, must be a constant rational map, which means that the curve $E$ is isotrivial. 
In positive characteristic, the isotriviality is also implied by Deuring's Theorem (see \cite[Chapter~13, \S6, Theorem~6.4]{Hus}). 
Let us give an \emph{ad hoc} proof of this lemma, which is independent of the characteristic of $K$. 
		\begin{proof} 
		The $j$-invariant $j(E)$ being non-constant, there exists a place $v$ of $K$ at which $j(E)$ has a pole (\ie{}, for which $v(j(E))<0$).
		Let $K_v$ denote the completion of $K$ at~$v$; the field $K_v$ is of the type considered in subsection \S\ref{ssec:tate.unif}.
		Since $v(j(E))<0$, we know from the results recalled there that there is a unique $t\in K_v$ with $v(t)>0$ such that 
		$E/K_v$ becomes isomorphic to the Tate curve $\E_t$ over a finite extension of~$K_v$. 
		Through this isomorphism, an endomorphism $\psi:E\to E$ induces an endomorphism $\Psi : \E_t\to\E_t$. 
		By \eqref{eq:endom.tate.unif}, there exists an integer $n$ such that $\Psi$ is the multiplication-by-$n$ map of $\E_t$.
		The original endomorphism $\psi : E\to E$ is thus nothing else but the multiplication map $[n]:E\to E$. Hence the result.
		\end{proof}

\section{Heights of elliptic curves}\label{sec:heights.EC} 

Let $k$ be a perfect field and $C$ be a smooth projective and geometrically irreducible curve over $k$. 
We let $K:=k(C)$ be its function field. 
We let $p\geq 0$ denote the characteristic of $K$ ($p$ is then either $0$ or a prime).

\subsection{Differential  and stable heights}\label{ssec:hdiff}

Let $L$ be a finite field extension of $K$; we choose a model $C'$ of~$L$ and write $k'$ for the field of constants of $L$. 
For an elliptic curve   $E$ over $L$, its minimal regular model $\Ecal$ 
 is the unique (up to isomorphism) smooth projective and geometrically irreducible surface over $k'$, 
equipped with a minimal surjective morphism $\pi:\Ecal\to C'$ whose generic fiber is $E$.
We denote by $\pi:\Ecal\to C'$ the minimal regular model of $E$ and $s_0:C'\to\Ecal$ its zero section.
We refer to \cite[Lecture 3, \S1-\S2]{UlmerParkCity} for more details about the construction of this model. 
Let $\Omega^1_{\Ecal/C'}$ be the sheaf of relative differential $1$-forms on $\Ecal$.
Pulling-back $\Omega^1_{\Ecal/C'}$ along the zero section~$s_0$ results in a line bundle on $C'$, 
which will be denoted by $\omega_{E/L} := s_0^\ast\Omega^1_{\Ecal/C'}$.
One  then defines the \emph{differential height of $E/L$} by
\[ \hdiff(E/L) :=   \frac{\deg\omega_{E/L}}{[L:K]},\]
where $\deg$ here means  degree of a line bundle on $C'$. 
The following identity is well-known:
		\begin{lemm}\label{lemm:silverman.hdiff}
		In the above setting, denote by  $\Deltamin(E/L)\in\Div(C')$ the minimal discriminant divisor of~$E$. 
		Then one has
		\begin{equation}\notag{}
		12\, \hdiff(E/L) = \frac{\deg\Deltamin(E/L)}{[L:K]}.
		\end{equation}
		\end{lemm}

This may be proved just as Proposition 1.1 in \cite{Silverman_HEC}: the main point is that $\Deltamin(E/L)$ provides 
a section of $\omega^{\otimes 12}_{E/L}$.
In contrast to the number field setting (treated in \cite{Silverman_HEC}) however, 
no ``Archimedean terms'' come into play in the computation of the degree of the line bundle $\omega_{E/L}$.

One checks that the differential height does not increase in finite extensions, by which we mean that, 
if $L'/L$ is a further finite extension, one has $\hdiff(E\times_L L' /L') \leq \hdiff(E/L)$. 
We refer the reader to \cite[Proposition 2.3, page 228]{MB} for a proof of this fact for general Abelian varieties. 
\\

Given an elliptic curve $E$ over $L$, as was recalled in \S\ref{ssec:semistab.red} above, 
we may find a finite extension $L'/L$ such that the base-changed curve $E\times_L{L'}/L'$ is semi-stable. 
We then define the \emph{stable height of $E$} to be
\[ \hstab(E/L) 
:=   \hdiff(E\times_{L}L'/L') =\frac{\deg(\omega_{E\times_L{L'}/L'})}{[L':K]}.\]
This definition makes sense: it is indeed shown in \cite[Proposition 2.3, page 228]{MB} that 
the quantity~$\hstab(E/L)$ does not depend on the choice of a particular extension $L'/L$ 
over which $E$ attains semi-stable reduction.

\subsection{Modular height}\label{ssec:hmod} 

Let $E$ be an elliptic curve defined over $\bar{K}$. 
Its $j$-invariant $j(E)$ lies in~$\bar{K}$: 
we can then define the \emph{modular height of $E$} to be
\begin{equation*}
\hmod(E) := h\big(j(E)\big) \in \Q_{\geq 0},
\end{equation*}
where $h$ is the Weil height on $\bar{K}$ defined in \S\ref{ssec:hweil}.
This modular height is closely related to the height called 
\emph{hauteur modulaire num\'erique} defined by Moret-Bailly in \cite[page 226]{MB}. 

If we fix a finite extension  $L$ of $K$ containing $j(E)$, and write $L=k'(C')$ 
(where $k'/k$ is a finite extension and $C'$ is a smooth projective geometrically irreducible curve  over $k'$), 
we may view $j(E)$ as a morphism $j(E):C'\to\P^1$ defined over $k'$. 
Denoting by $\div_\infty(j(E))\in\Div(C')$ the divisor of poles of this map, we then have
\[\hmod(E)=\frac{\deg( \div_\infty(j(E)))}{[L:K]}.\]
(See the discussion in~\S\ref{ssec:hweil}). 
It follows that the modular height $\hmod(E)$ vanishes if and only if 
$j(E)$ is a constant element of $\bar{K}$ \ie{}, if and only if $E$ is isotrivial.

\subsection{Comparison of heights}\label{ssec:comp.heights} 
 
We now compare the various notions of heights of an elliptic curve over~$\bar{K}$ that were just introduced.
Let $L/K$ be a finite extension, and let $E$ be an elliptic curve over $L$.
		\begin{prop}\label{prop:compare.hdiff.hmod}	
		In the above setting, one has
		\begin{equation}\label{eq.comparaison.hdiff.hmod}	
		0  \leq \hdiff(E/L) -\frac{\hmod(E)}{12} 
		\leq \frac{1}{[L:K]}  \, \sum_{w \text{ \rm not s.s.}} \deg w,
		\end{equation}
		where the sum is over the (finite) set of places of $L$ where $E$ does not have semi-stable reduction.
		\end{prop}

The reader might want to compare this estimate with the one in \cite[Proposition 2.1]{Silverman_HEC} 
where a similar comparison is carried out between differential and modular heights of an elliptic curve over a number field.
The proof in our setting is simplified by the absence of Archimedean places.

		\begin{proof} 
		For any place $w$ of $L$, we let $\delta_w\geq0$ denote the valuation at $w$ of the discriminant 
		of a minimal integral model of $E$ at $w$, 
		and we let  $\theta_w = w( j(E))$ denote the order of the pole/zero of $j(E)$ at $w$.

		By Lemma~\ref{lemm:silverman.hdiff} and the definition of the modular height, we have
		\begin{align*}
		[L:K] \, \hdiff(E/L)
		&= \frac{\deg\Deltamin(E/L)}{12}  = \frac{1}{12}\sum_{w\in M_L } \delta_w \, \deg w, \\
		[L:K] \, \hmod(E)
		&= \sum_{w\in M_L} n_w  \, \max\{0, \log |j(E)|_w\}  = \sum_{w\in M_L}  \max\{0, - \theta_w\} \deg w,
		\end{align*}
		where the sums are supported on the places of bad reduction of $E$. 
		Note indeed that the $j$-invariant has no poles outside places of bad reduction. 
		(Actually, as the table page\ 365 of \cite{ATAEC} shows, poles of $j$  only occur 
		at places of potentially multiplicative reduction.) Hence we have
		\[ 12 [L:K] \, \hdiff(E/L) - [L:K] \, \hmod(E) = \sum_{w\in M_L} \min\{\delta_w, \delta_w+\theta_w\}  \deg w,\] 
		the sum being supported on places where $E$ has bad reduction.
		If $w$ is a place of semi-stable reduction for $E$, then the Kodaira--N\'eron theorem (see \S\ref{ssec:semistab.red}) 
		implies that $\delta_w=-\theta_w$.
		If $w$ is a place of additive reduction, then a quick perusal at the above mentioned table in \cite{ATAEC} shows that 
		$\delta_w+\theta_w$ is contained in $\{0, \dots, 12\}$.  
		Therefore, for any place $w$ of bad reduction, the integer
		$ \min\{\delta_w, \theta_w+\delta_w\}$ lies in $[0,12]$; and it equals~$0$ if $E$ has semi-stable reduction at $w$. 
		We have thus proved that 
		\begin{equation*}
		0\leq 12 [L:K] \, \hdiff(E/L) - [L:K] \, \hmod(E) 
		=  \sum_{w \text{ bad red.}} \min\{\delta_w, \delta_w + \theta_w\} \deg w
		\leq 12 \sum_{w\text{ not s.s.}} \deg w.
		\end{equation*}
		This entails the desired bounds. 
		\end{proof}

It is clear that an elliptic curve $E/L$ is semi-stable if and only if 
the sum on the right-hand side of \eqref{eq.comparaison.hdiff.hmod} vanishes. 
Hence, the above proposition directly implies:

		\begin{coro}\label{coro:comp.hdiff.hmod.semistable}
		Let $E$ be an elliptic curve defined over a finite extension $L$ of $K$.
		Then $E/L$ is semi-stable  if and only if $\hmod(E) = 12 \, \hdiff(E/L)$. 
		In particular, we have $\hmod(E)=12 \, \hstab(E/L)$.
		\end{coro}

\subsection{Differential height and conductor}

For any elliptic curve $E$ over a finite extension $L$ of $K$, Ogg's formula 
(see formula (11.1) in \cite[Chapter IV, \S11]{ATAEC}), 
implies that $w(\cond({E/L}))\leq w(\Deltamin(E/L))$ for all places $w\in M_L$. 
Hence we have $\deg\cond(E/L) \leq \deg \Deltamin(E/L)$, and  it follows that 
\[\deg\cond (E/L) \leq 12 [L:K]\, \hdiff(E/L).\]
Obtaining an inequality in the other direction is much harder. 
The result is as follows:

		\begin{theo}[Szpiro's inequality] \label{theo:szpiro.inequality}
		Let $E$ be an elliptic curve defined over a function field $L$ of genus~$g(L)$. We have
		\begin{equation}\label{eq.szpiro.inequality}
		\deg\Deltamin(E/L) \leq 6\, \degins j(E) \, \big(2g(L)-2 + \deg\cond({E/L}) \big),
		\end{equation}	
		where $\cond(E/L)$ denotes the conductor of $E$, and $\degins j(E)$ the inseparability degree of its $j$-invariant.
		\end{theo}

The proof in the semi-stable case can be found in \cite{Szpiro90}, and the general case is treated in \cite{PesentiSzpiro}.
In particular, for an elliptic curve $E/K$, Szpiro's inequality reads:
\begin{equation}
\frac{\deg\cond(E/K) }{12}
\leq  \hdiff(E/K) 
\leq \degins j(E) \, \left(\frac{\deg\cond(E/K)}{2}  +  g(K)-1 \right).
\end{equation}
Hence, for elliptic curves over $K$ whose $j$-invariant is separable, the ``numerical conductor'' $\deg\cond(E/K)$ 
and the differential height $\hdiff(E/K)$ have the same order of magnitude (up to constants depending at most on $g(K)$).

\section{Isogenies between elliptic curves} 
\label{sec:isogenies}

The goal of this section is to recall a few standard facts about isogenies between elliptic curves, 
as well as prove decomposition results for them. 
The reader is referred to \cite[Chapter III, \S4]{AEC} for further details about isogenies.
We work in the same setting as before: 
we let $k$ be a perfect field,  $C$ be a smooth projective and geometrically irreducible curve over $k$, 
and $K$ be the function field $k(C)$.
We denote the characteristic of $K$ by $p$ (whether it is a prime or $0$), 
and fix an algebraic closure $\bar{K}$ of $K$. All algebraic extensions of $K$ are viewed as sub-extensions of $\bar{K}$.

\subsection{Preliminaries on isogenies}\label{ss:isogenies.prelim}

Let $L\subset\bar{K}$ be a  finite extension of $K$.
Let $E_1$ and $E_2$ be two elliptic curves defined over $L$. 
We denote by $0_{E_1}\in E_1(L)$ and $0_{E_2}\in E_2(L)$ the respective neutral elements of the groups $E_1$ and $E_2$.
An \emph{isogeny} $\varphi: E_1\to E_2$ is a surjective morphism of varieties satisfying $\varphi(0_{E_1})=0_{E_2}$. 
It can then be shown that $\varphi$ is a group morphism $E_1\to E_2$.
(Unless otherwise specified, we do not assume that isogenies are defined over $L$.)
As a morphism between algebraic varieties, $\varphi$ induces a finite embedding of function fields
\[ \varphi^\ast:\bar{K}(E_2)\hookrightarrow \bar{K}(E_1).\]
The \emph{degree of $\varphi$}, denoted by $\deg\varphi$, is defined to be 
the degree of the finite field extension $\bar{K}(E_1)/\varphi^\ast(\bar{K}(E_2))$. 

If the characteristic of $K$ is positive, 
the extension $\bar{K}(E_1)/\varphi^\ast(\bar{K}(E_2))$ may be split into two subextensions, as follows:
we let $M_\varphi$ be the separable closure of $\varphi^\ast(\bar{K}(E_2))$ in $\bar{K}(E_1)$.  
Then $M_\varphi/\varphi^\ast(\bar{K}(E_2))$ is a finite separable extension, whose degree is denoted by $\degsep\varphi$ 
and called the \emph{separable degree of~$\varphi$}.
The degree of the extension  $\bar{K}(E_1)/M_\varphi$ is denoted by $\degins\varphi$ 
and is called the \emph{inseparable degree of~$\varphi$}.
It is clear that  $\degins\varphi$ is a power of the characteristic of $K$, and that we have 
$\deg\varphi=\degsep\varphi \, \degins\varphi$.
\\

Coming back to the general case, and viewing $\varphi : E_1\to E_2$ as a homomorphism of group varieties, 
we may  define its kernel $\ker\varphi=\varphi^{-1}(\{0_{E_2}\}) \subset E_1$ as a a finite group variety over $L$. 
As a group scheme, the kernel of $\varphi$ is reduced if and only if $\varphi$ is separable 
(in which case we do not lose much by identifying the kernel with its set of closed points). 
In general, one can check that the finite Abelian group $(\ker\varphi)(\bar{K})$ has order $\degsep\varphi$. 
A separable isogeny is called \emph{cyclic} if its kernel is a cyclic Abelian group. 

If $\varphi : E_1\to E_2$ is an isogeny, recall that there is a \emph{dual isogeny} $E_2\to E_1$, which we denote by $\hat{\varphi}$.
The compositions $\varphi \circ\hat{\varphi}$ and $\hat{\varphi}\circ \varphi$ are equal 
to the multiplication by $\deg\varphi$ on $E_2$, respectively $E_1$. 
The degrees of $\varphi$ and $\hat{\varphi}$ are equal.
\\

We denote the set of all isogenies $E_1\to E_2$ together with the constant morphism~$0_{E_2}:E_1\to E_2$ by $\Hom(E_1, E_2)$. 
This set is naturally endowed with the structure of a $\Z$-module.
If $E_1$ and $E_2$ are non-isotrivial and isogenous, one can show that 
the $\Z$-module $\Hom(E_1, E_2)$ is torsion-free of rank $1$ (since they have no CM, by \S\ref{ssec:CM.isotriv}); 
thus, there exists an isogeny $\varphi_0:E_1\to E_2$ so that $\Hom(E_1, E_2) = \Z \, \varphi_0$. 
Such a~$\varphi_0$ is unique up to multiplication by $[\pm1]$, and is characterised (up to sign) by the fact that 
it has minimal degree among all non-zero elements of $\Hom(E_1, E_2)$.
\\

Let us now assume that $K$ has positive characteristic $p$, and let $E$ be an elliptic curve over $L$. 
For any power $q$ of $p$, we write $E^{(q)}$ for the $q$-th power Frobenius twist of $E$: 
if $E$ is defined by a Weierstrass model with coefficients $a_1, a_2, a_3, a_4, a_6\in L$, 
then $E^{(q)}$ is the elliptic curve given by the Weierstrass model with coefficients $a_1^q, a_2^q, a_3^q, a_4^q,  a_6^q$
(note that $j(E^{(q)})=j(E)^q$).
Recall that there is a \emph{$q$-th power Frobenius morphism} $\Frob{q}:E\to E^{(q)}$ which is defined over $L$, and that 
this morphism is an isogeny of degree $q$.
As such, it admits a dual isogeny $\Vers{q}:E^{(q)}\to E$ which is called the \emph{$q$-th power Verschiebung isogeny}, 
and is also defined over $L$.
The multiplication-by-$q$ map $[q]:E\to E$ then decomposes as $[q] = \Vers{q}\circ \Frob{q}$.

If $E$ is non-isotrivial, it is known that $\Frob{q}:E\to E^{(q)}$ is purely inseparable of degree $q$ 
(that is, we have $\degsep\Frob{q} = 1$ and $\degins\Frob{q} = q$) and that its dual $\Vers{q}:E^{(q)} \to E$ is separable of degree $q$. 
Note that, since $E$ is non-isotrivial, it is \emph{ordinary} in the sense that: $E[q]\simeq \Z/q\Z$ for any power $q$ of $p$, 
where $E[q]$ denotes the subgroup of $q$-torsion elements of $E(\bar{K})$.
These facts follow from  \cite[Chapter V, \S3]{AEC}.
\\

Finally, we recall three results concerning isogenies between elliptic curves.
		\begin{prop}\label{prop:decomp.isog.pi.sep} 
		Let $E_1, E_2$ be two elliptic curves defined over a field $L$, whose characteristic $p$ is positive,  
		and let $\varphi : E_1\to E_2$ be an isogeny between them. 
		If one writes $q=\degins\varphi$, 
		then the isogeny $\varphi$ factors as $\varphi =\psi \circ \Frob{q}$ 
		where $\Frob{q}:E_1\to E_1^{(q)}$ is the $q$-th power Frobenius isogeny and $\psi : E_1^{(q)}\to E_2$ is a separable isogeny. 
		\end{prop}

		\begin{prop}\label{prop:decomp.isog.factor}
		Let $E_1, E_2, E_3$ be elliptic curves defined over a field $L$.
		Let $\varphi :E_1\to E_2$ be a separable isogeny, and $\psi :E_1\to E_3$ be an isogeny.
		If $\ker\varphi\subseteq\ker\psi$, then $\psi$ factors uniquely through $\varphi$ \ie{}, there is a unique isogeny $\lambda:E_2\to E_3$ such that $\psi = \lambda\circ\varphi$.
		\end{prop}

		\begin{prop}\label{prop:isog.quotient}
		Let $E$ be an elliptic curve defined over a field $L$, and let $G$ be a finite subgroup of $E$.
		Then there exist a unique elliptic curve $E'$ defined over $\bar{L}$ and a separable isogeny $\pi: E \to E'$ such that $\ker\pi=G$.
		The curve $E'$ is usually denoted by $E/G$.
		\end{prop}

These three statements and their proofs can be found in  \cite{AEC}: 
see Corollary~2.12 in Chapter II,
Corollary~4.11 in Chapter III, 
and Proposition~4.12 in Chapter III, respectively.

\subsection{Biseparable isogenies} 

It will be convenient to consider the following class of isogenies separately:
 	\begin{defi}\label{defi:ds.isog}
	Let $\varphi:E\to E'$ be an isogeny between two elliptic curves $E$, $E'$ defined over $\bar{K}$. 
	We will say that $\varphi$ is \emph{biseparable} if both~$\varphi$ and its dual $\hat{\varphi}$ are separable. 
	We will say that two elliptic curves $E, E'$ over~$\bar{K}$ are \emph{biseparably isogenous} if 
	there exists a biseparable isogeny between them.
	\end{defi}

If $K$ has characteristic $0$, all isogenies between elliptic curves over $\bar{K}$ are of course biseparable. 
In positive characteristic, such isogenies may be characterised as follows: 
		\begin{lemm}\label{lemm:isog.ds.degree}
		Let $E$ and $E'$ be two elliptic curves over $\bar{K}$ and let $\varphi:E\to E'$ be an isogeny.
		Then $\varphi$ is biseparable if and only if $\deg\varphi$ is coprime to the characteristic of $K$. 
		\end{lemm}	
		\begin{proof}
		Let $p>0$ denote the characteristic of $K$. 
		We factor the degree $d:=\deg\varphi$ as $d=p^r d'$, with $r\geq 0$ and $d'\in\Z$  coprime to $p$.
		By construction of the dual isogeny, we have 
		$\hat{\varphi}\circ\varphi = [d] = [d']\circ[p^r]$, where $[n]:E\to E$ denotes the multiplication-by-$n$ map on $E$.
		Now, the multiplication-by-$p^r$ map $[p^r]$ on $E$ is inseparable of degree $p^{2r}$ 
		(with inseparability degree  $\degins[p^r]\in\{p^r,p^{2r}\}$ depending on whether~$E$ is ordinary or not);
		and, since $d'$ is coprime to $p$, the map $[d']:E\to E$ is separable. 
		All in all, the inseparability degree of the map $\hat{\varphi}\circ\varphi = [d]:E\to E$ satisfies 
		$\degins[d]\in\{p^r,p^{2r}\}$.
		
		Now assume that $\varphi$ is biseparable of degree $d=p^rd'$.
		Since both $\varphi$ and $\hat{\varphi}$ are separable, their composition is separable. 
		Hence the multiplication-by-$d$ map on $E$ is separable.
		By the previous paragraph, we must have $r=0$. 
		Therefore, the degree of $\varphi$ 	is coprime to $p$.
		
		Conversely, assume  that the degree $d$ of $\varphi$ is coprime to $p$. 
		Then $\varphi$ must be separable because the map $[d]:E\to E$ is separable and factors through $\varphi$. 
		Since the dual of $\varphi$ has degree $\deg\hat{\varphi}=d$,
		the  same argument shows that $\hat{\varphi}$ is also separable.
		\end{proof}

We also note the following:
		\begin{lemm}\label{lemm:kernel.sep}
		Let $K$ be a function field as above, and $E_1, E_2$ be two non-isotrivial elliptic curves 
		over a finite extension~$L$ of  $K$.
		Let $\varphi:E_1\to E_2$ be a biseparable isogeny between them.
		Then the kernel $H=(\ker\varphi)(\bar{K})$ of $\varphi$ is defined over a finite separable extension of $L$, 
		that is, the extension $L(H)/L$ is separable.
		\end{lemm}
		\begin{proof} 
		Let $d$ be the degree of $\varphi$,  
		$[d]:E_1\to E_1$ denote the multiplication-by-$[d]$ map on $E_1$, 
		and $E_1[d]$ be the subgroup of $E_1(\bar{K})$ formed by $d$-torsion points.
		We have $H = (\ker \varphi)(\bar{K}) \subset (\ker[d])(\bar{K}) = E_1[d]$, so that $L(H)$ is a subfield of $L(E_1[d])$.
		It thus suffices to prove that the extension $L(E_1[d])/L$ is separable.
		We know by Lemma~\ref{lemm:isog.ds.degree} that~$d$ is coprime to the characteristic of $L$.
		It is then known (see \cite[\S1]{SerreTate} or \cite[Proposition 3.8]{BanLonVig} for instance) 
		that $E_1[d]$ is contained in $E_1(L\sep)$, where $L\sep$ denotes the separable closure of~$L$ in $\bar{K}$.  
		The extension $L(E_1[d])/L$ is therefore separable, and the lemma is proved.
		\end{proof}

\subsection{Useful decompositions of isogenies} 

Let $K$ be a function field in the above sense. 
In this subsection, we assume that the characteristic $p$ of $K$ is positive.

Let $E_1, E_2$ be two non-isotrivial elliptic curves defined over $\bar{K}$, 
and let $\varphi:E_1\to E_2$ be an isogeny between them. 
By Proposition \ref{prop:decomp.isog.pi.sep}, one can decompose $\varphi$ as
\[E_1\xrightarrow{\Frob{p^e}} E_1^{(p^e)} \xrightarrow{\psi} E_2,\]
where $\Frob{p^e}$ denotes the $p^e$-th power Frobenius isogeny, and $\psi$ is a separable isogeny.
Since $\psi$ is separable, we see that $\degins\varphi= \degins \Frob{p^e}=p^e$. 
Let us now consider the dual isogeny $\hat{\psi} : E_2\to E_1^{(p^e)}$ to $\psi$.
Using the same Proposition \ref{prop:decomp.isog.pi.sep}, 
we obtain that $\hat{\psi}$ factors as
\[E_2\xrightarrow{\Frob{p^f}} E_2^{(p^f)} \xrightarrow{\psi'} E_1^{(p^e)},\]
where $\psi'$ is a separable isogeny, and where $p^f = \degins \hat{\varphi}$. 
We thus have $\hat{\psi} = \psi'\circ\Frob{p^f}$. 
By contravariance of taking duals, we deduce that $\psi =\widehat{\widehat{\psi}} = \widehat{\Frob{p^f}}\circ\widehat{\psi'}$.
By definition, the dual of $\Frob{p^f}$ is the Verschiebung isogeny $\Vers{p^f} : E_2^{(p^f)}\to E_2$.
Since $E_2$ is non-isotrivial, $\Vers{p^f}$ is a separable isogeny of degree~$p^f$.
The isogeny $\varphi_s:=\hat{\psi'}$ is separable, because both $\psi = \Vers{p^f}\circ\varphi_s$ and $\Vers{p^f}$ are 
(by multiplicativity of the inseparability degree in compositions).
In particular, $\varphi_s$ is a separable isogeny whose dual $\hat{\varphi_s} =\psi'$ is also separable.

We have therefore decomposed our original isogeny $\varphi:E_1\to E_2$ as a composition
\[ E_1 \xrightarrow{\Frob{p^e}} E_1^{(p^e)} \xrightarrow{\varphi_s} E_2^{(p^f)} \xrightarrow{\Vers{p^f}} E_2,\]
where $p^e=\degins\varphi$ and $p^f =\degins\hat{\varphi}$, and where $\varphi_s$ is biseparable. 

This discussion proves the following decomposition result:
		\begin{prop}\label{prop:decomposition.isogeny}
		Let $E_1, E_2$ be two non-isotrivial isogenous elliptic curves  over $\bar{K}$, 
		and let $\varphi : E_1\to E_2$ be an isogeny between them.
		The isogeny $\varphi$ factors as
		\[ E_1 
		\xrightarrow{\ \Frob{p^e}\ } E_1^{(p^e)} 
		\xrightarrow{\ \psi\ } E_2^{(p^f)} 
		\xrightarrow{\ \Vers{p^f}\ } E_2, \]
		where $\psi$ is a biseparable isogeny, $p^e = \degins\varphi$ and $p^f= \degins\hat{\varphi}$.	
		\end{prop}

The above statement will be used repeatedly in the remainder of the article. 
Let us also prove now the following factorisation result: 
		\begin{prop}\label{prop:isogeny.reduction}
		Let $\varphi : E_1\to E_2$ be an isogeny between two non-isotrivial elliptic curves over $\bar{K}$,
		and write $q := \min\{\degins\varphi, \degins\widehat{\varphi}\}$.
		Then the isogeny $\varphi$ factors through the multiplication-by-$q$ map $[q]$.
		\end{prop}
		\begin{proof} 
		Up to replacing $\varphi$ by its dual and exchanging the roles of $E_1, E_2$, we may assume that $q = \degins\varphi$. 
		It suffices to treat this case because $\varphi$ factors through $[q]$ if and only if $\hat{\varphi}$ does. 
		We begin by decomposing~$\varphi$ as in Proposition \ref{prop:decomposition.isogeny} above: 
		there is a biseparable isogeny $\psi:E_1^{(p^e)}\to E_2^{(p^f)}$ such that 
		\[\varphi = \Vers{p^f} \circ \psi \circ \Frob{p^e},  \]
		where $\degins\varphi=p^e$ and $\degins\hat{\varphi}=p^f$, and where
		the degree $m$ of $\psi$ is coprime to $p$.
		By assumption we have $q=p^e \leq p^f$. 		
		Note that $\deg\varphi = p^emp^f = q^2 p^{d} m$ where $d:=f-e\geq 0$, and that $\degsep\varphi = m p^f$.

		We note that it is enough to treat the case where $e=f$.  
		Consider indeed the isogeny $\varphi': E_1\to E_2^{(p^d)}$ defined by $\Vers{p^e}\circ\psi\circ\Frob{p^e}$: 
		if $\varphi'$ factors through $[q]$ then so does $\varphi$, because $\varphi = \Vers{p^d}\circ \varphi'$.
		We furthermore note that the result is immediate if $f=0$, we thus assume that $f\geq 1$.
		This reduces the problem to the case where $e=f\geq 1$ (\ie{}  $d=0$ and $\varphi=\varphi'$), 
		which is now assumed for the rest of the proof.
		\\

		The set $G:=(\ker\varphi)(\bar{K})$ of $\bar{K}$-rational points in the kernel of $\varphi$ 
		is a finite Abelian group of order $\degsep\varphi = mp^f$. 
		(Note that, as a group scheme, the kernel $\ker\varphi$ is not reduced.) 
		Since $p^f$ divides~$|G|$, 
		$G$~contains a unique subgroup $H$ of order $p^f$ 
		(which consists of  its elements of order a power of $p$, since $m$ is coprime to $p$).
		By Proposition \ref{prop:isog.quotient}, there exist a unique elliptic curve $E'_1$ over $\bar{K}$ 
		and a separable isogeny $\pi : E_1\to E'_1$ with kernel $H$.
		The isogeny $\pi$ has degree $\deg\pi = \degsep\pi=p^f$.

		By construction, $H=(\ker\pi)(\bar{K})$ is a subgroup of $G$ and, $\pi$ being separable, 
		we deduce from Proposition~\ref{prop:decomp.isog.factor} that $\varphi:E_1\to E_2$ factors through $\pi$: 
		there exists an isogeny $\lambda: E'_1\to E_2$ such that $\varphi = \lambda\circ\pi$.
		It follows from the multiplicativity of degrees that $\degsep\lambda =m$ and $\degins\lambda =p^e$.
		We use Proposition \ref{prop:decomposition.isogeny} to decompose $\pi$.
		Since $\pi$ is separable of degree a power of $p$, we obtain that the diagram 
		\begin{center}\begin{tikzcd}
		E_1    
		\ar[r, " \pi "] 
		\ar[d,  "\iota"'] & E'_1 \\
		(E'_1)^{(p^f)}
		\ar[ur , "\Vers{p^f}"']
		\end{tikzcd}\end{center}
		is commutative, where $\iota$ is a biseparable isogeny. 
		Comparing degrees, we observe that $\iota$ has degree $1$ and 
		must therefore be an isomorphism $E_1\simeq (E'_1)^{(p^f)}$.
		 
		Let us apply Proposition \ref{prop:decomposition.isogeny} once more, this time to factor $\lambda$.
		Since the separability degree of $\lambda$ is coprime to $p$, 
		we deduce that $\hat{\lambda}$ is separable.
		Therefore, there is a biseparable isogeny $\gamma : (E'_1)^{(p^e)}\to E_2$ such that $\lambda = \gamma\circ\Frob{p^e}$. 
		Here is a diagram summarising the various isogenies considered here: 
		\begin{center}\begin{tikzcd}[row sep=large, column sep=large, transform shape, nodes={scale=1.0}]
		&  E_1 \arrow[r, "\varphi"]
		\arrow[d, "\pi"] 
		\arrow[ld, "\iota", "\text{\rotatebox{30}{$\simeq$}}"'] 
		& E_2 & \\
		(E'_1)^{(p^f)} 
		\arrow[r, "\Vers{p^f}"'] 
		\arrow[rr, bend right=40, "{[p^e]}"] 
		& E'_1 
		\arrow[r, "\Frob{p^e}"'] 
		\arrow[ru, "\lambda"] 
		& (E'_1)^{(p^e)} 
		\arrow[u, "\gamma"]& 
		\arrow[l, "\text{\rotatebox{0}{$\simeq$}}"' , "\iota"] 
		\arrow[lu, "\phi"', dashrightarrow] E_1. 
		\end{tikzcd}\end{center}
		\noindent 
		In this diagram, the triangles are commutative (recall that we assume $e=f$).		

		We now set $\phi := \gamma\circ\iota$. 
		The composition $\phi : E_1\to E_2$ is an isogeny between $E_1$ and $E_2$, 
		and we have $\varphi = \phi \circ \hat{\iota}\circ \Frob{p^e}\circ \Vers{p^e}\circ\iota$. 
		But we know that $\Frob{p^e}\circ\Vers{p^e}=[p^e]=[q]$, and we have proved that $\hat{\iota}\circ\iota =\mathrm{id}$. 
		This shows that $\varphi = \phi\circ[q] = [q] \circ \phi$. 
		Hence, $\varphi$ factors through $[q]$ in the case at hand, and thus, in general. 
		\end{proof}

\begin{rema}\label{rema:isog.min.sep}
Let  $E_1, E_2$ be two non-isotrivial isogenous elliptic curves. 
As was mentioned in \S\ref{ss:isogenies.prelim},  the $\Z$-module $\Hom(E_1, E_2)$ 
is free of rank $1$: we may thus fix an isogeny $\varphi_0:E_1\to E_2$ which generates $\Hom(E_1, E_2)$. 
This generator $\varphi_0$ has minimal degree among all non-zero elements in $\Hom(E_1, E_2)$.
By the above proposition and the minimality of $\deg\varphi_0$, 
one has $\min\{\degins\varphi_0, \degins\widehat{\varphi_0}\} = 1$. 

This means that at least one of $\varphi_0$ or $\widehat{\varphi_0}$ is separable.
\end{rema}

\section{Isogenies and heights}\label{sec:isog.heights} 

Let $k$ be a perfect field of characteristic $p\geq 0$.
Let $C$ be a smooth projective geometrically irreducible curve over $k$, and $K=k(C)$ denote its function field. 
We fix an algebraic closure $\bar{K}$ of $K$.

The goal of this section is to describe the effect of an isogeny between elliptic curves over $\bar{K}$ on the height of their $j$-invariants:
we prove Theorem~\ref{itheo:comp.hmod} (see \S\ref{ss:isog.hmod}), as well as state a few consequences thereof.
The general idea is that ``biseparable isogenies preserve the modular height''. 
If $K$ has characteristic $0$, this will directly lead to the desired result.
In positive characteristic $p$, more work is required.

\subsection{Biseparable isogenies preserve the differential height}\label{ss:ds.isog.hdiff} 

We begin by recalling the following result due to Parshin (see \cite{Parshin70, Parshin71} for instance) 
as well as giving a proof thereof, for the convenience of the reader. 
The reader is referred to \cite{BLR} (especially \S7.3 there) and \cite{Raynaud} for more details about N\'eron models and isogenies. 

		\begin{theo}\label{theo:isog.preserve.hdiff}
		Let $E_1, E_2$ be non-isotrivial elliptic curves over a finite extension $L$ of $K$.
		Assume that there exists a biseparable isogeny $\varphi: E_1\to E_2$. 
		Then, we have 
		\[\hdiff(E_1/L)=\hdiff(E_2/L).\]
		\end{theo}
		\begin{proof} 
		We fix a model $C'/k'$ of $L$.
		For $i=1,2$, we denote the N\'eron model of $E_i/L$ by $\pi_i : \Ecal_i\to C'$, 
		and we write $s_i:C'\to\Ecal_i$ for its zero-section.
		The surface $\Ecal_i$ is smooth  over $C'$: 
		let $\Omega^1_{\Ecal_i/C'}$ denote the sheaf of relative differential $1$-forms on $\Ecal_i/C'$, and
		consider the line bundle $\omega_i := s_i^\ast\Omega^1_{\Ecal_i/C'}$ on~$C'$. 
		Recall from~\S\ref{ssec:hdiff} that the differential height of $E_i/L$ essentially  equals $\deg\omega_i$.
		To prove the theorem, it is (more than) sufficient to show that $\omega_1\simeq \omega_2$ as line bundles on~$C'$.
		The given isogeny $\varphi:E_1\to E_2$ extends into a group morphism $\Phi:\Ecal_{1} \to \Ecal_{2}$ which is still an isogeny, 
		which means that $\Phi$ is, fiber by fiber, finite and surjective on the identity components. 
		The morphism~$\Phi$ induces a map $\Phi^\ast\Omega^1_{\Ecal_2/C'} \to \Omega^1_{\Ecal_1/C'}$, 
		which we may restrict to the zero-section. 
		Using that $\Phi$ satisfies $\Phi\circ s_1 = s_2$, we obtain a map of $\O_{C'}$-modules 
		\[F :  \omega_2 = s_2^\ast\Omega^1_{\Ecal_2/C'} \simeq s_1^\ast\Phi^\ast\Omega^1_{\Ecal_2/C'} 
		\longrightarrow s_1^\ast\Omega^1_{\Ecal_1/C'} = \omega_1.\]
		Is suffices to show that $F$ is an isomorphism: to do so, we may argue locally on $C'$.
		We thus set out to show that, for any closed point $w$ of $C'$, 
		the restriction of $F$ to the fibers $\omega_{1,w}$ and $\omega_{2,w}$ of $\omega_1$ and $\omega_2$ above~$w$ is an isomorphism.

		Let $w$ be a closed point of $C'$. Write $\O_{w}$ for the local ring of $C'$ at $w$ and $S:=\Spec \O_w$. 
		We also let $k_w$ denote the residue field at $w$ (note that $k_w$ is a finite extension of $k'$, 
		and thus has the same characteristic as $k'$).
		Denoting the fiber of $\Ecal_i$ at $w$ by $\Ecal_{i,w} := \Ecal_i \times_{C'} \Spec\O_w$, 
		the restriction $\varphi_w : \Ecal_{1,w}\to \Ecal_{2,w}$ of $\Phi$ is an isogeny. 
		It is then known (see \cite[\S7.3, Lemma 5]{BLR}) that there exists a dual isogeny $\hat{\varphi}_w:\Ecal_{2,w}\to \Ecal_{1,w}$, 
		such that $\hat{\varphi}_w\circ\varphi_w=[d]$ with $d=\deg\varphi$. 
		Since $d$ is coprime to the characteristic of $k_w$, all three of $\varphi_w$, $\hat{\varphi}_w$ and $[d]$ are \'etale 
		by \cite[\S1.1, Exemples 1.1.2]{Raynaud}, or  \cite[\S7.3, Lemma 2]{BLR}.

		The isogeny $\varphi_w$ induces a canonical map 
		$\varphi_w^\ast\Omega^1_{\Ecal_{2,w}/S}\to \Omega^1_{\Ecal_{1,w}/S}$ which,
		since $\varphi_w$ is \'etale on $\Ecal_{1,w/S}$,   is actually an isomorphism (see \cite[\S2.2, Corollary 10]{BLR}). 
		Moreover, $\varphi_w$ being a group morphism, we have $\varphi_w\circ s_1|_{\Ecal_{1,w}} = s_2|_{\Ecal_{2,w}}$.
		Pulling back  the above isomorphism along (the restriction of) $s_1$, we obtain an isomorphism of $\O_{w}$-modules
		\[ \omega_{2,w} 
		= s_2^\ast\Omega^1_{\Ecal_{2, w}/S} \simeq s_1^\ast\varphi_w^\ast\Omega^1_{\Ecal_{2,w}/S}  
		\xrightarrow[]{\phantom{..}\simeq\phantom{..}}  s_1^\ast\Omega^1_{\Ecal_{1,w}/S} = \omega_{1,w},\]
		which is the restriction of $F$ to the fibers above $w$. This concludes the proof.
		\end{proof}

\subsection{Biseparable isogenies preserve the inseparability degree of the $j$-invariants}\label{ss:ds.isog.deginsj} 

In this subsection, we assume that the function field $K$ has positive characteristic $p$. 
We describe the effect of a biseparable isogeny on the inseparability degree of the $j$-invariants.

		\begin{prop}\label{prop:ds.iso.insepdeg.j} 
		Let $E_1, E_2$ be two \emph{non-isotrivial} elliptic curves over a finite extension $L$ of $K$. 
		If there exists a biseparable isogeny $E_1 \to E_2$, 
		then the inseparability degrees of $j(E_1)$ and $j(E_2)$ are equal.
		\end{prop}

The proof below is adapted from \cite[2.2]{BanLonVig} where it is proven that, 
if $E_1\to E_2$ is a biseparable isogeny and if $j(E_1)\in K$ is separable, then $j(E_2)$ is separable too.

		\begin{proof} 
		Fix a model $C'/k'$ for $L$.
		We write $j_1$, $j_2$ to denote the $j$-invariants of $E_1$ and $E_2$ respectively, and
		let $\varphi:E_1\to E_2$ be a biseparable isogeny, whose degree is denoted by $d$ 
		(by Lemma \ref{lemm:isog.ds.degree}, $d$ is thus coprime to $p$).
		Choose an elliptic curve $E'_2$ defined over $k'(j_2)$ whose $j$-invariant is $j_2$:
		by construction, $E_2$ and $E'_2$ are isomorphic over $\bar{K}$. 
		Actually, $E_2$ and $E'_2$ become isomorphic over a finite separable extension $L'/L$ of degree $\leq 24$, 
		see \cite[Chapter III, Proposition 1.4]{AEC}.
		Moreover one can assume that the field of constants of $L'$ is $k'$.
		We may and do base change the situation to $L'$, so that $E_2\simeq E'_2$ over $L'=L$.

		For the convenience of the reader, here is a diagram of the field extensions we are going to consider: 
		all these extensions are finite, we have indicated some of the inseparability degrees, 
		and the dashed lines denote extensions which are  separable (see below). 
		
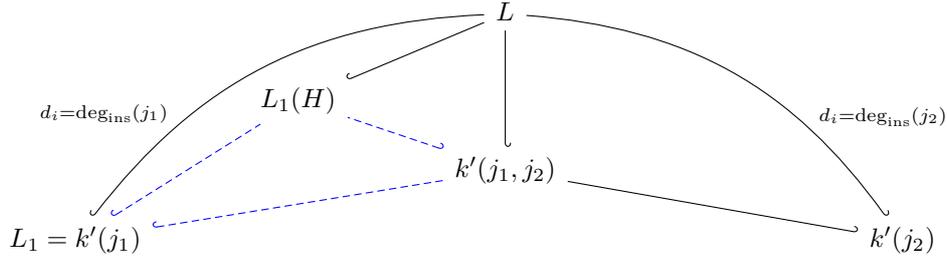
\begin{center}\begin{tikzcd}[row sep = small, column sep=large, transform shape, nodes={scale=1.0}]
		&& L && \\ \\
		& L_1(H) 
		\ar[uur, hook, no head]
		&& \phantom{L_2(H)}& \\
		&& k'(j_1, j_2) 
		\ar[uuu, hook', no head]
		\ar[ul, hook', blue,  dashed, no head]
		 && \\
		L_1=k'(j_1)
		\ar[uuuurr, hook, no head, bend left= 25, "d_i= \degins(j_1)" near start]
		\ar[uur, hook, blue, dashed, no head]
		\ar[urr, hook, blue,  dashed, no head]
		&&&& k'(j_2) 
		\ar[uuuull, hook', no head, bend right= 25, "d_i= \degins(j_2)"' near start]
		\ar[ull, hook', no head]
		\end{tikzcd}\end{center}

		Let $L_1=k'(j_1)\subset L$.
		We view $\varphi$ as an isogeny $\varphi:E_1\to E'_2$. 
		By Lemma \ref{lemm:kernel.sep}, the group $H:=(\ker{\varphi})(\bar{K})$ of $\bar{K}$-rational points in the kernel of ${\varphi}$ 
		is then defined over a separable extension of $L$: the (finite) extension $L_1(H)/L_1$ obtained 
		by adjoining to $L_1$ the coordinates of points of $H$ is therefore separable.

		Because the curves $E_1$ and $E'_2$ are linked by the isogeny ${\varphi}$, we have $E'_2\simeq E_1/H$. 
		In particular, $E_2$ is isomorphic to an elliptic curve defined over $L_1(H)$, so that $j_2\in L_1(H)$.
		This implies that ${L_1(j_2) = k'(j_1, j_2)}$, being a field sub-extension of $L_1(H)/L_1$, is a finite separable extension of $L_1$. 
		Hence the extension $L/L_1$ has the same inseparability degree as the  extension $L/k'(j_1,j_2)$: 
		\[ \degins(j_1) = [L:L_1]_i = [L: k'(j_1, j_2)]_i.\] 
		On the other hand, we have a chain of finite extensions $k'(j_2) \subset k'(j_1, j_2) \subset L$. 
		Therefore, $[L:k'(j_1,j_2)]_i$ divides $[L:k'(j_2)]_i = \degins(j_2)$.
		We thus conclude that $[L:k'(j_1)]_i = \degins(j_1)$ divides $\degins(j_2)$. 

		Applying the same argument to the dual isogeny $\hat{\varphi}:E'_2\to E_1$ (which is also biseparable), 
		we conclude that $\degins(j_2)$ divides $\degins(j_1)$. 
		These two quantities are therefore equal.
		\end{proof}

Making use of Proposition \ref{prop:decomposition.isogeny}, 
we now describe the effect of an arbitrary isogeny on the inseparability degrees of the $j$-invariants.
		\begin{coro}\label{coro:dualdeg}
		Let $E_1, E_2$ be two non-isotrivial elliptic curves defined over a finite extension $L$ of $K$. 
		Assume that there exists an isogeny $\varphi : E_1\to E_2$.   
		Then we have 
		\[{\degins\hat{\varphi}}\ \, \degins j(E_2)
		=  {\degins\varphi}\ \, \degins j(E_1).\]
		\end{coro}
		\begin{proof} 
		By Proposition \ref{prop:decomposition.isogeny}, one may decompose $\varphi:E_1\to E_2$ as a composition
		\[ E_1 
		\xrightarrow{\ \Frob{p^e}\ } E_1^{(p^e)} 
		\xrightarrow{\ \psi \ }E_2^{(p^f)} 
		\xrightarrow{\ \Vers{p^f}\ } E_2, \]
		where $\psi$ is biseparable,  $p^e=\degins\varphi$, and $p^f = \degins\hat{\varphi}$.
		By the previous proposition, the inseparability degrees of $j(E_1^{(p^e)})$ and $j(E_2^{(p^f)})$ are equal.
		For  $i\in\{1, 2\}$ and any $a\in \Z_{\geq0}$, since $j(E_i^{(p^a)}) = j(E_i)^{p^a}$, 
		we have $\degins j(E_i^{(p^a)})  = p^a \, \degins(j(E_i))$.
		The claimed identity is then clear. 
		\end{proof}

\subsection{Isogenies and modular heights}\label{ss:isog.hmod} 

We can now conclude our study of the effect of isogenies between elliptic curves on their modular height. 
We fix a function field $K$ as above.
In case  $K$ has characteristic $0$, all the inseparability degrees should be interpreted as being $1$.
The following result was announced in the introduction as Theorem \ref{itheo:comp.hmod}:

		\begin{theo} \label{theo:isog.hmod}
		Let $E_1$ and $E_2$ be two non-isotrivial elliptic curves defined over $\bar{K}$. 
		Assume that there is an isogeny $\varphi : E_1\to E_2$  between them. 
		Then one has
		\[ \hmod(E_2) = \frac{\degins\varphi}{\degins\hat{\varphi}} \ \hmod(E_1).\]
		In particular, biseparable isogenies preserve the modular height.
		\end{theo}
		\begin{proof} 
		Fix a finite extension $L$ of $K$ over which both $E_1$, $E_2$ are defined. 
		We may and will also assume that $L$ is chosen so that $E_1$ and $E_2$ are semi-stable over $L$.
		As above (see Proposition \ref{prop:decomposition.isogeny}), 
		we factor the isogeny $\varphi : E_1\to E_2$  as a composition of 
		$\Frob{p^e}:E_1\to E_1^{(p^e)}$, $\psi : E_1^{(p^e)}\to E_2^{(p^f)}$ 
		and $\Vers{p^f} :E_2^{(p^f)}\to E_2$ where $\psi$ is biseparable and $p^e=\degins\varphi$, $p^f=\degins\hat{\varphi}$. 
		For simplicity, we denote $E_1^{(p^e)}$ by $E'_1$ and $E_2^{(p^f)}$ by $E'_2$.
		We write that 
		\[ \frac{\hmod(E_2)}{\hmod(E_1)} 
		= \frac{\hmod(E_2)}{\hmod(E'_2)} \  \frac{\hmod(E'_2)}{\hmod(E'_1)} \  \frac{\hmod(E'_1)}{\hmod(E_1)}. \]
		By definition, we have $j(E'_1) = j(E_1^{(p^e)}) = j(E_1)^{p^e}$, so that 
		$\hmod(E'_1)=p^e \, \hmod(E_1)$ and, similarly, we have $\hmod(E'_2) = p^f \,  \hmod(E_2)$. 
		We now use the ``semi-stable case'' of Proposition \ref{prop:compare.hdiff.hmod} and obtain that  
		\[ \hmod(E'_2) - \hmod(E'_1) 
		\stackrel{(i)}{=} 12 \, \hdiff(E'_2/L) -  12\, \hdiff(E'_1/L)
		\stackrel{(ii)}{=} 0.\]
		Here, equality $(i)$ follows from Proposition \ref{prop:compare.hdiff.hmod}: 
		the ``error term'' in the comparison of heights vanishes because the curves are semi-stable over $L$; 
		equality $(ii)$ comes from the fact that $\psi:E'_1\to E'_2$, being biseparable, 
		preserves the differential height (see Theorem \ref{theo:isog.preserve.hdiff}).

		Therefore $\hmod(E'_1)=\hmod(E'_2)$, and the result is proved.
		\end{proof}

\begin{rema}\label{rema:alternative.approach} 
An alternative (and  rather natural) approach to proving Theorem~\ref{theo:isog.hmod} would be to rely 
on considering the modular polynomial relations satisfied by the $j$-invariants of isogenous elliptic curves. 
We now explain why the most immediate properties of these modular relations  are insufficient to conclude. 

The crucial step towards 
Theorem \ref{theo:isog.hmod} consists in showing that \emph{biseparable cyclic isogenies preserve the modular height}.  
Let, thus, $j_1, j_2\in\bar{K}$ be the $j$-invariants of two non-isotrivial elliptic curves 
linked by a biseparable cyclic isogeny.
Fix a finite extension $L=k'(C')$ of $K$ containing $j_1$ and $j_2$, and 
a biseparable cyclic isogeny $\varphi$ between two elliptic curves over $L$ 
with respective $j$-invariants $j_1$ and $j_2$. 

Writing $N:=\deg\varphi$ for the degree of $\varphi$, 
we then have $\Phi_N(j_1, j_2)=0$ in~$L$, where $\Phi_N(X,Y)\in\Z[X,Y]$ 
denotes the $N$-th modular polynomial.
Recall from (see \cite[Chapter V, \S2]{LangEllFun} that~$\Phi_N(X,Y)$  has coefficients in $\Z$, 
is symmetric in $X,Y$, and that it is monic of degree $\psi(N)$ in each variable. 

In this situation, using integrality and symmetry of the modular relation $\Phi_N(j_1, j_2)=0$,
one can show that the divisors of poles $\div_\infty(j_1)$, $\div_{\infty}(j_2)$ have the same support. 
One cannot, however, conclude that $h(j_1)=h(j_2)$ from these elementary results on $\Phi_N(X,Y)$ alone,
as illustrated by the following example.
For any integer $d\geq 2$, consider the polynomial 
${P_d(X,Y) := X^{d+1} + Y^{d+1} - (XY)^{d} -XY}\in\Z[X,Y]$:
it too is symmetric in $X,Y$, and is monic of degree $d+1$ in each variable.
For any $y_1\in L\smallsetminus k'$, we let $y_2:=y_1^d$.
We then have $P_d(y_1, y_2)=0$, even though $y_2$ has height $h(y_2)= d \,  h(y_1) > h(y_1)$.
	
Therefore, any attempt at proving Theorem~\ref{theo:isog.hmod}   through this route must make use of finer results on~$\Phi_N(X,Y)$, 
such as irreducibility or arithmetic facts about its coefficients (see \cite{Cohen} or \cite[\S4]{Paz18}). 
We think it unlikely that these considerations lead to a simpler proof of Theorem \ref{theo:isog.hmod} in arbitrary characteristic.
\end{rema}

\subsection{Isogenies and differential heights}\label{ss:isog.hdiff} 

As we saw in \S\ref{ss:ds.isog.hdiff}, 
biseparable isogenies preserve the differential height (Theorem~\ref{theo:isog.preserve.hdiff}). 
If $K$ has characteristic $0$, all isogenies are biseparable so that Theorem~\ref{theo:isog.preserve.hdiff} completely solves the problem of 
describing the effect of an arbitrary isogeny on the differential height.
In this subsection, we thus assume that the function field $K$ has positive characteristic $p$ 
and attempt to describe this effect. 
In view of the decomposition given by Proposition~\ref{prop:decomposition.isogeny}, 
it is enough to focus on the effect of the Frobenius isogeny on $\hdiff$: 
this is what we elucidate now. 

		\begin{lemm}\label{lemm:hdiff}
		Let $L$ be a finite extension of $K$.
		For any non-isotrivial elliptic curve $E$ over $L$, there exists $\alpha(E/L) \geq 1$ such that, 
		for any power $q$ of $p$, we have 
		\begin{equation}\label{prems}
		\alpha(E/L)^{-1} \, q \, \hdiff(E/L) \leq\hdiff(E^{(q)}/L)\leq q \, \hdiff(E/L).
		\end{equation}
		If, moreover, $E$ is semi-stable over $L$, we have $\alpha(E/L)=1$, so that $\hdiff(E^{(q)}/L) = q \, \hdiff(E/L)$.
		\end{lemm}
		\begin{proof} 
		To lighten notation, we write $E':=E^{(q)}$.
		For any place $w$ of $L$, we denote the  ring of integers at~$w$ by $\O_w\subset L$.
		We may pick a minimal $w$-integral Weierstrass model for $E$ of the form:
		\[E: \quad y^2+a_{1, w}xy+a_{3, w}y=x^3+a_{2, w}x^2+a_{4, w}x+a_{6, w},\]
		with $a_{1,w}, a_{2,w}, a_{3,w}, a_{4,w}, a_{6,w}\in \O_w$.
		By definition, the discriminant 
		\[\Delta_{E,w} = \Delta(a_{1,w}, a_{2,w}, a_{3,w}, a_{4,w}, a_{6,w})\in \O_w\]
		of this model has minimal valuation $w(\Delta_{E,w})\in\Z_{\geq 0}$ 
		among all choices of $w$-integral Weierstrass coefficients $a_{1,w}$, $a_{2,w}$, $a_{3,w}$, $a_{4,w}$, $a_{6,w}$ for $E$.
		A Weierstrass model for the Frobenius twist  $E'$ of $E$ is then given by:
		\begin{equation}\label{eq:Wmod.p}
		E': \quad y^2+a_{1, w}^{q}xy+a_{3, w}^{q}y=x^3+a_{2, w}^{q}x^2+a_{4, w}^{q}x+a_{6, w}^{q}.	
		\end{equation}
		We note that $a_{1, w}^{q}, a_{2, w}^{q}, a_{3, w}^{q}, a_{4, w}^{q}, a_{6, w}^{q}$ all lie in $\O_w$, 
		so that \eqref{eq:Wmod.p} is   a $w$-integral model for $E'$.
		The discriminant $\Delta'_w$ of this model is clearly equal to $\Delta_{E,w}^{q}$. 
		Let $\Delta_{E', w}$ be the discriminant of a minimal $w$-integral Weierstrass model of $E'$.
		The difference $w(\Delta'_w) -w(\Delta_{E', w})$ is then a non-negative integral multiple of $12$, 
		for $\Delta'_w$ and $\Delta_{E', w}$ differ by the $12$-th power of an element of $\O_w$.
		In particular, we obtain that $q \, w(\Delta_{E, w})\geq w(\Delta_{E',w})$.
		Multiplying this inequality by $\deg w$, and summing over all places $w$ of $L$ yields that 
		\[q \, \deg\Delta_{\min}(E/L) \geq \deg\Delta_{\min}(E'/L).\]
		The right-most  inequality in \eqref{prems} ensues immediately.
		To prove the other inequality in \eqref{prems}, we argue as follows.
		The proof of Proposition \ref{prop:compare.hdiff.hmod} implies the bounds:
		\begin{equation}\label{eq:pouet}
		0 \leq \deg\Delta_{\min}(E/L) - \deg\div_\infty(j(E)) \leq 12 \deg \Acal(E/L),
		\end{equation}
		where $\Acal(E/L)=\sum_{w\text{ not s.s.}} w$ is the divisor whose support consists in 
		the places $w$ of $L$ where $E$ does not have semi-stable reduction. 
		In particular, we have
		\[ \deg \Delta_{\min}(E'/L) \geq \deg\div_\infty(j(E')) \qquad \text{ and } \qquad 
		\deg\Delta_{\min}(E/L) \leq  \deg \div_\infty(j(E)) + 12 \deg \Acal(E/L).\]
		Therefore, since $j(E')= j(E)^{q}$, we deduce that
		\begin{align*}
		\frac{q \, \deg\Delta_{\min}(E/L)}{\deg\Delta_{\min}(E'/L)} 
		& \leq \frac{q  \, \deg\Delta_{\min}(E/L)}{\deg\div_\infty(j(E'))}
		= \frac{q \, \deg\Delta_{\min}(E/L)}{q \, \deg\div_\infty(j(E))} \\
		&\leq \frac{ \deg \div_\infty(j(E)) + 12 \deg \Acal(E/L)}{\deg\div_\infty(j(E))}
		= \alpha(E/L),
		\end{align*}
		where we have set $\alpha(E/L) := 1+ 12\deg \Acal(E/L)/(\deg\div_\infty j(E))$. 
		From which we obtain that 
		\[q \, \hdiff(E/L) \leq \alpha(E/L) \, \hdiff(E'/L).\]
		It is clear that $\alpha(E/L)\geq 1$, and that $\alpha(E/L)=1$ if and only if $E/L$ is semi-stable.
		The above proves both the left-most inequality in \eqref{prems} and the last assertion of the Lemma.
		\end{proof}

\begin{rema}\label{rema:isog.hdiff.charnot23} 
We assume here that $K$ has characteristic $p\neq 2, 3$.
Let $E$ be a non-isotrivial elliptic curve over $L$, and $E':=E^{(q)}$ be its $q$-th power Frobenius twist. 
Since $E$ and $E'$ are $L$-isogenous, they have the same reduction behaviour at all places of $L$.

If  $w$ is a place  where $E$ has multiplicative reduction of type $\mathbf{I}_n$, for some $n\geq 1$, 
then $E'$ also has multiplicative reduction at $w$, and its fiber at $w$ is of type $\mathbf{I}_{n  \, q}$. 
Therefore we have 
\[w(\Delta_{\min}(E/L)) \leq q \, w(\Delta_{\min}(E/L)) =q \, n = w(\Delta_{\min}(E'/L)).\]
If $w$ is a place where $E$ has additive reduction, then  
by inspection of the possible Kodaira--N\'eron reduction types of $E'$ at $w$ 
(see the table page\ 365 of \cite{ATAEC}, which is only valid in characteristic $p\neq 2, 3$), 
we find  that $w(\Delta_{\min}(E/L)) \leq 12 \, w(\Delta_{\min}(E'/L))$. 
Therefore, for any place $w$ where either of $E$ and $E'$ have bad reduction,  we have  
$w(\Delta_{\min}(E/L)) \leq 12 \, w(\Delta_{\min}(E'/L))$.
Multiplying this inequality by $\deg w$ and summing over all places $w$ of $L$, we obtain that
\[\frac{1}{12} \, \hdiff(E/L)\leq \hdiff(E^{(q)}/L).\] 
This may be viewed as a weak but uniform version of the lower bound in \eqref{prems}.	
\end{rema}

We can now give the final estimate of this subsection.
		\begin{prop}\label{prop:iso.gen} 
		Let $E_1, E_2$ be a pair of \emph{non-isotrivial} elliptic curves over a finite extension $L$ of $K$. 
		Assume that there exists an isogeny $\varphi: E_1 \to E_2$. 
		If both $E_1$  and $E_2$ are semi-stable over $L$, we have
		\[ \hdiff(E_2/L)  = \frac{\degins\varphi}{\degins\hat{\varphi}} \ \hdiff(E_1/L).\]
		In general, we have 
		\[ \alpha(E_2/L)^{-1}
		\leq \frac{\degins\varphi}{\degins\hat{\varphi}} \ \frac{\hdiff(E_1/L)}{\hdiff(E_2/L)}
		\leq \alpha(E_1/L),\]
		where $\alpha(E_1/L), \alpha(E_2/L)\geq 1$ are the same as in Lemma \ref{lemm:hdiff}.
		\end{prop}
		\begin{proof}
		The semi-stable case is a direct consequence of 
		Theorem \ref{theo:isog.hmod} and Corollary \ref{coro:comp.hdiff.hmod.semistable}.
		In the general case, by Proposition \ref{prop:decomposition.isogeny},
		the isogeny $\varphi:E_1\to E_2$ decomposes as 
		\[ \varphi: E_1 \xrightarrow{\ \Frob{p^e}\ } E_1^{(p^e)} \xrightarrow{\ \psi \ }E_2^{(p^f)} \xrightarrow{\ \Vers{p^f}\ } E_2, \]
		where $\psi$ is a biseparable isogeny,  $p^e=\degins\varphi$, and $p^f = \degins\hat{\varphi}$. 
		We then write that
		\[ \frac{\hdiff(E_1/L)}{\hdiff(E_2/L)} 
		= \frac{\hdiff(E_1/L)}{\hdiff(E_1^{(p^e)}/L)} \ 
		\frac{\hdiff(E_1^{(p^e)}/L)}{\hdiff(E_2^{(p^f)}/L)} \ 
		\frac{\hdiff(E_2^{(p^f)}/L)}{\hdiff(E_2/L)}. \]
		Since the isogeny $\psi:E_1^{(p^e)}\to E_2^{(p^f)}$ is biseparable, 
		Theorem~\ref{theo:isog.preserve.hdiff} yields that ratio of ${\hdiff(E_1^{(p^e)}/L)}$ to ${\hdiff(E_2^{(p^f)}/L)}$ is $1$. 
		To conclude, one applies inequality \eqref{prems} in Lemma \ref{lemm:hdiff} to the other two ratios.
		\end{proof}

\begin{rema} 
If $K$ has characteristic $p\neq 2, 3$, 
we may carry out the same argument using the bound in Remark \ref{rema:isog.hdiff.charnot23} instead of inequality \eqref{prems}; 
we would then obtain the following bound. 
If $\varphi:E_1 \to E_2$ is an isogeny between two non-isotrivial elliptic curves $E_1, E_2$ over $L$, we have
\[(12)^{-1}\leq \frac{\degins\varphi}{\degins\hat{\varphi}} \ \frac{\hdiff(E_1/L)}{\hdiff(E_2/L)}\leq 12.\]
\end{rema}

\subsection{A surprising consequence on isogeny classes}\label{ss:isog.classes} 

Let $E$ be a non CM elliptic curve over $\bar{\Q}$. 
For any $B\geq0$, consider the set
\[\Jscr_\Q(E, B) 
= \big\{E'/\bar{\Q} :  E' \text{ is isogenous to } E \text{ and } ht(j(E'))\leq B \big\}\big/\,\bar{\Q}\text{-isomorphism},\]
where $ht:\bar{\Q}\to\R$ is the standard logarithmic absolute Weil height on $\bar{\Q}$. 
It is known that $\Jscr_\Q(E, B)$ is a finite set (see Lemma~5.9 in \cite{Habegger}); 
the main input in the proof of that statement is the estimate from Th\'eor\`eme~1.1 in \cite{SzpiroUllmo} which asserts that 
\[ ht(j(E')) \geq ht(j(E)) + \frac{1}{2}\log\deg\varphi - o(\log\deg\varphi) \]
if there is a cyclic isogeny $\varphi : E \to E'$ (where the implicit constants in the error term depend on $E$). 
Using our results in \S\ref{ss:isog.hmod}, we study the analogue to $\Jscr_\Q(E, B)$ in the context of function fields.
\\

Let $K$ be a function field as in section \ref{sec:intro.FF}, with characteristic $p\geq 0$.
Fix a non-isotrivial elliptic curve~$E$ over $K$, and write $\Iscr_{\mathrm{bs}}(E)$ for the set of elliptic curves over $\bar{K}$ 
which are biseparably isogenous to $E$ (\ie{}, such that there is a biseparable isogeny $E\to E'$). 
When the characteristic of $K$ is $0$, all isogenies are biseparable, 
so that $\Iscr_{\mathrm{bs}}(E)$ is the whole isogeny class of $E$.
For any real number $B\geq 0$, consider the set
\begin{align*}
\Jscr_K(E, B) &= \big\{E'\in\Iscr_{\mathrm{bs}}(E) : \hmod(E')\leq B \big\}\big/\,\bar{K}\text{-isomorphism} \\	
&= \big\{j(E'), \ E'\in\Iscr_{\mathrm{bs}}(E) \text{ with } \hmod(E')\leq B \big\}.
\end{align*}

The biseparability condition in $\Iscr_{\mathrm{bs}}(E)$ is added in order to avoid trivial situations in positive characteristic: 
without this condition, if $K$ has characteristic $p$, we would indeed directly obtain from Theorem~\ref{theo:isog.hmod} 
infinitely many (isomorphism classes of) elliptic curves over $\bar{K}$ which are isogenous to $E$ and 
have bounded modular height by considering the sequence
\[E \xrightarrow[]{\ \Vers{p} \ } E^{(1/p)} \xrightarrow[]{\ \Vers{p} \ }  \dots 
\xrightarrow[]{\ \Vers{p} \ } E^{(1/p^{n-1})}\xrightarrow[]{\ \Vers{p} \ } E^{(1/p^{n})} \xrightarrow[]{\ \Vers{p} \ } \dots.\]

If $B< \hmod(E)$, Theorem~\ref{theo:isog.hmod} implies that $\Jscr_K(E, B)=\varnothing$.
If $B\geq \hmod(E)$, in stark contrast to the above mentioned result of Habegger, we prove:
		\begin{prop}\label{prop:infinite}
		If $B\geq \hmod(E)$, the set $\Jscr_K(E, B)$ is infinite.	
		\end{prop}
		\begin{proof}
		Given an integer $n\geq 1$ which is coprime to the characteristic $p$ of $K$, 
		we may pick a  point $P_n\in E(\bar{K})$ of exact order $n$.
		We then let $\pi_n$ denote the quotient isogeny $\pi_n:E\to E/\langle P_n\rangle$ from $E$ to its quotient~$E_n$ 
		by the subgroup generated by $P_n$.
		The isogeny $\pi_n$ has degree $n$ (see Proposition \ref{prop:isog.quotient}), 
		and is therefore biseparable (see Lemma~\ref{lemm:isog.ds.degree}). 

		This construction provides a sequence $(E_n)_{n}$ of elliptic curves over $\bar{K}$ 
		indexed by prime-to-$p$ integers. 
		It is clear that the curve $E_n$ is biseparably isogenous to $E$.
		Applying  Theorem \ref{theo:isog.hmod} to the isogeny $\pi_n : E\to E_n$ yields that $\hmod(E_n)=\hmod(E)$.
		Therefore, the isomorphism class of $E_n$ does belong to $\Jscr_K(E, B)$, thanks to our assumption that $\hmod(E)\leq B$. 

		In order to conclude the proof, it now suffices to show that 
		the above-constructed sequence $(E_n)_{n}$ provides infinitely many isomorphism classes. 
		Let $\ell_1,\ell_2$ be two prime numbers, both coprime to $p$.  
		Assume that $E_{\ell_1}\simeq E_{\ell_2}$ and denote the isomorphism by $\iota: E_{\ell_1}\to E_{\ell_2}$.
		The composition $\iota\circ\pi_{\ell_1}$ is an isogeny $E\to E_{\ell_2}$ of degree $\ell_1$. 
		Hence, the isogeny $\psi : E\to E$ defined by $\psi = \widehat{\pi_{\ell_2}}\circ\iota\circ\pi_{\ell_1}$ 
		is a biseparable endomorphism of $E$ of degree $\ell_1\ell_2$.
		Since $E$ is non-isotrivial, its endomorphism ring is trivial (see \S\ref{ssec:CM.isotriv}); 
		hence, there is an integer $d$ such that $\psi=[d]$.
		Taking degrees, we obtain that $d^2=\ell_1\ell_2$.
		Since $\ell_1$ and $\ell_2$ are primes, we deduce that $\ell_1=\ell_2=d$.
		We thereby conclude that the sequence $(E_\ell)_{\ell}$ 
		indexed by prime numbers $\ell\neq p$ yields infinitely many elements in $\Jscr_K(E, B)$.
		\end{proof}

The main difference between the number field and the function field cases lies in 
how much the modular height varies along an isogeny class.
Over number fields, the above-mentioned Th\'eor\`eme 1.1 in \cite{SzpiroUllmo} shows that 
$ht(j(E'))$ \emph{does} vary as $E'/\bar{\Q}$ runs through the isogeny class of $E/\bar{\Q}$, 
provided that $E$ has no~CM. 
Therefore, bounding $ht(j(E'))$ sufficiently constrains the degree of possible isogenies $E\to E'$ that 
the set $\Jscr_\Q(E,B)$ is finite.
Over a function field $K$, on the contrary, our Theorem \ref{theo:isog.hmod} shows that 
$\hmod(E')$ remains constant as $E'$ runs through $\Iscr_{\mathrm{bs}}(E)$, provided that $E$ is non-isotrivial.
The absence of constraints on the degree of isogenies $E\to E'$ other 
than coprimality with the characteristic, allows the set $\Jscr_K(E,B)$ to be infinite.
In this situation we will, in \S\ref{ss:back.isog.classes}, 
recover a finiteness statement by further imposing a bound on $[K(j(E'))=K]$ for $E'\in\Jscr_K(E,B)$.

\section{An isogeny estimate}\label{sec:isog.estimate} 

The primary goal of this last section is to prove the second main result of the paper (Theorem \ref{itheo:min.isog}), 
which is the following isogeny estimate:

		\begin{theo}\label{theo.small.isogenies} 
		Let $K$ be a function field of genus $g$. 
		For any pair of  non-isotrivial isogenous elliptic curves  $E_1, E_2$ defined over $K$, 
		there exists an isogeny $\varphi_0:E_1\to E_2$ with
		\[ \deg\varphi_0
		\leq \cO{} \max\{1,  g\} \, \max\left\{\frac{\degins j(E_1)}{\degins j(E_2)}, \frac{\degins j(E_2)}{\degins j(E_1)}\right\},\]
		where $\degins j(E_1), \degins j(E_2) $ are the inseparability degrees of $j(E_1), j(E_2)$, respectively.
		\end{theo}

If $K$ has characteristic $0$, the inseparability degrees appearing on the right-hand side of the inequality should be interpreted as~$1$. 
Specifically, in that situation, the statement above yields  
that there is an effective constant $c_1>0$ (depending only on the genus of $K$) such that: 
for all pairs $E_1, E_2$ of non-isotrivial $\bar{K}$-isogenous elliptic curves, 
there exists a ($\bar{K}$-)isogeny $\varphi_0:E_1\to E_2$ with $\deg\varphi_0 \leq c_1$.
This isogeny estimate is thus \emph{uniform} for a fixed $K$. 
This should be compared to the number field case (treated in \cite{MasserW1, Pellarin, GaudronRemond}) 
where the right-hand side of such isogeny estimates depends on the heights of the involved elliptic curves.

Let us  remark that, in positive characteristic $p$, the appearance of the inseparability degrees on the right-hand side is unavoidable,
because of Theorem \ref{theo:isog.hmod} and  the existence of the Frobenius isogeny. 
Given a non-isotrivial elliptic curve  $E/K$, for any integer $n\geq 1$, the smallest isogeny between $E$ and its Frobenius twist $E^{(p^n)}$ is 
indeed the $p^n$-th power Frobenius, which has degree $p^n$.
Similar considerations with the Verschiebung show (with Corollary \ref{coro:dualdeg})
that the dependency in $\degins j(E_1)$ and $\degins j(E_2)$ on the right-hand side of the bound in Theorem \ref{theo.small.isogenies} is optimal. 

On the other hand, the dependency of Theorem \ref{theo.small.isogenies} in the genus of $K$ can sometimes be improved. 
For instance, if $K$ has genus $0$ one can prove (see Remark \ref{rema:genusbound}) that
there exists an isogeny $\varphi_0:E_1\to E_2$ with
\begin{equation}
\deg\varphi_0 \leq 25  \,  \max\left\{\frac{\degins j(E_1)}{\degins j(E_2)}, \frac{\degins j(E_2)}{\degins j(E_1)}\right\}
\end{equation}
between two isogenous non-isotrivial elliptic curves $E_1, E_2$ defined over $K$.
  
\begin{rema}\label{rema:field.of.def} 
If one assumes that $E_1, E_2$ are linked by an isogeny   which is defined over $K$, then
the isogeny $\varphi_0:E_1\to E_2$ whose existence is asserted in Theorem \ref{theo.small.isogenies} is also defined over $K$. 
More generally, all isogenies $E_1\to E_2$ are then defined over $K$. 
See Lemma \ref{lemm:field.of.def} below. 
\end{rema}

\subsection{Preliminaries about modular curves}\label{ss:mod.curves} 

Let us briefly recall some facts about modular curves. 
Given an integer $N\geq 1$, there is a smooth scheme $\Ycal_0(N)$ of relative dimension $1$ over $\Spec \Z[1/N]$ 
which is a coarse moduli scheme for elliptic curves endowed with a cyclic subgroup of order~$N$. 
In other words,  the curve $\Ycal_0(N)$ enjoys the following property:
For any field $F$ whose characteristic is coprime to $N$, there is a bijection (which is functorial in $F$) between 
the set $\Ycal_0(N)(F)$ and the set of equivalence classes of pairs $(E, H)$ 
where $E$ is an elliptic curve over $F$, 
and $H$ is a cyclic subgroup of $E$ of order $N$ which is stable under the action of the absolute Galois group of $F$. 
Two such pairs are called equivalent if they are isomorphic over the algebraic closure $\bar{F}$ of $F$.
These properties of $\Ycal_0(N)$ are explained in more details in \cite[\S8]{DiamondIm}, 
and the construction is carried out comprehensively in \cite{KaMa} (in particular, Chapters III,  VI and VIII there).
Adding a finite number of points (called cusps) to $\Ycal_0(N)$, one obtains the usual compactification, 
denoted by $\Xcal_0(N)$,  of $\Ycal_0(N)$. 
The smooth projective curve $\Xcal_0(N)$ is also defined over $\Spec\Z[1/N]$ 
(and, further, has an interpretation as a moduli space, in terms of generalized elliptic curves, see \cite[\S9]{DiamondIm}).

In particular, the curve $\Xcal_0(1)$ is nothing but the ``$j$-line'' over $\Spec\Z$ 
\ie{}, $\Xcal_0(1) \simeq \P^1_{/\Z}$, 
where the isomorphism is given on isomorphism classes of elliptic curves by $E\mapsto j(E)$.
 
For any divisor $n$ of $N$, there is a ``degeneration'' morphism $\Xcal_0(N) \to \Xcal_0(n)$, 
which extends the map on pairs $(E,H)$ as above, defined by $(E,H)\mapsto (E, \frac{N}{n} H)$. 
In the special case $n=1$, we obtain a map $f_N:\Xcal_0(N) \to \Xcal_0(1) = \P^1$, 
which extends the map $(E, H)\mapsto E \mapsto j(E)$.
\\
 
Let $\Gamma_0(N)\subset \SL_2(\Z)$ be the subgroup formed by $2\times 2$ 
integral matrices of determinant $1$ whose reduction modulo~$N$ is upper triangular. 
The fiber $X_0(N) := \Xcal_0(N)\times\C$ is isomorphic (as a Riemann surface) 
to the compactification of the quotient $\{\tau\in\C:\mathrm{Im}\,\tau>0\}/\Gamma_0(N)$.
One can show (see \cite[Chapter I]{Shimura} for instance) that 
the degree of the degeneration morphism $f_N:X_0(N) \to X_0(1) = \P^1_{/\C}$ is equal to the index~${[\mathrm{SL}_2(\Z):\Gamma_0(N)]}$.
By Propositions 1.40 and 1.43 in \cite{Shimura}, we have
\begin{equation}\label{eq:degree.degeneration}
\deg f_N = {[\mathrm{SL}_2(\Z):\Gamma_0(N)]}   = N \, \prod_{\ell\mid N} \left(1+\frac{1}{\ell}\right):= \psi(N).
\end{equation}
A detailed study of the ramification behaviour of $f_N$, 
combined with the Riemann--Hurwitz formula then allows to compute the genus of $X_0(N)$.
Specifically, Proposition 1.40 and Proposition 1.43 in  \cite{Shimura} yield that 
the modular curve $X_0(N)$ has genus
\begin{equation}\label{eq:genus.modular}
g({X_0(N)}) = 1+\frac{\psi(N)}{12} - \frac{\nu_2(N)}{4}- \frac{\nu_3(N)}{3} - \frac{\nu_\infty(N)}{2},
\end{equation}
where $\psi(N)$ is as defined in \eqref{eq:degree.degeneration},  
$\nu_\infty(N) = \sum_{d\mid N} \varphi(\gcd(d, N/d))$,  and
\[\nu_2(N) = \begin{cases} 0 & \text{ if } 4\mid N, \\
 \prod_{\ell\mid N}\left(1+\left(\frac{-1}{\ell}\right)\right)	& \text{ if } 4\nmid N,
 \end{cases}\qquad \text{ and }\quad 
 \nu_3(N) = \begin{cases} 0 & \text{ if } 9\mid N, \\
 \prod_{\ell\mid N}\left(1+\left(\frac{-3}{\ell}\right)\right)	& \text{ if } 9\nmid N.
 \end{cases}\]
Here $\varphi$ denotes Euler's totient function, and $(\cdot/\ell)$ is the Legendre symbol.
\\

We end this brief summary by providing explicit bounds on the genus $g(X_0(N))$ of $X_0(N)$ as $N$ grows. 
We thank Ga\"el R\'emond for the following argument, which yields  better bounds than our original proof.

		\begin{lemm}\label{lemm:genus.estimate}
		For any integer  $N\geq 300$, we have $ g({X_0(N)}) \geq {N}/{49}$. 
		Moreover, for all $N\geq 1$, we have \[ N \leq \cO{}\max\{1, g({X_0(N)})\}.\]
		\end{lemm}
		\begin{proof} 
		We  estimate separately the  terms on the right-hand side of equation \eqref{eq:genus.modular}.
		We start by noticing that $\nu_\infty, \nu_2$ and $\nu_3$ are all multiplicative functions.
		For any prime $\ell$, and any integer $e\geq1$, one has 
		\[\nu_\infty(\ell^e) 
		= \begin{cases}
		\sqrt{\ell ^e} \left(1+\ell^{-1}\right) & \text{ if } 2\mid e, \\
		\sqrt{\ell ^e} \, 2 \ell^{-1/2} & \text{ if } 2\nmid e.
		\end{cases}\]
		Since, for all $\ell$, we have $1+\ell^{-1}\geq {2} {\ell^{-1/2}}$, the multiplicativity of $\nu_\infty$ implies that, for all $N\geq 1$,
		\[ \frac{\nu_\infty(N)}{\sqrt{N}}
		\leq \prod_{\ell \mid N}\left(1+\frac{1}{\ell}\right) 
		= \frac{\psi(N)}{N}.\]
		Hence  we deduce that $\nu_\infty(N)\leq {\psi(N)}/{\sqrt{N}}$.

		Similarly, we notice that  $\max\{\nu_2(\ell^e), \nu_3(\ell^e)\}\leq\nu_\infty(\ell^e)$ for all primes $\ell$ and integers $e\geq 1$. 
		We directly infer that, for any integer $N\geq1$, one has $\max\{\nu_2(N), \nu_3(N)\}\leq\nu_\infty(N)$.
		 
		Plugging these bounds into equation \eqref{eq:genus.modular} yields that
		\begin{align*}
		g({X_0(N)})
		&\geq1+\frac{\psi(N)}{12}-\frac{7}{12}\max\{\nu_2(N), \nu_3(N)\}-\frac{1}{2}\nu_\infty(N)\\
		&\geq1+\frac{\psi(N)}{12}-\frac{13}{12}\nu_\infty(N)
		\geq 1+\frac{\psi(N)}{12}\left(1-\frac{13}{\sqrt{N}}\right).
		\end{align*}
		By definition of $\psi$, we have $\psi(N)\geq N$.
		Therefore, $g({X_0(N)}) \geq 1 +  {(N-13\sqrt{N})}/{12}$ for $N\geq 13^2$.
		Furthermore, as a quick computation shows,  for any $N\geq 297$, we have
		\[ g({X_0(N)}) 
		> \frac{N}{12}\left(1-\frac{13}{\sqrt{N}}\right)
		\geq \frac{N}{49}.\]
		This proves the first assertion of the lemma.
		Direct calculations with formula \eqref{eq:genus.modular}  using a computer show that the bound 
		$N\leq \cO{} \max \{1, g({X_0(N)})\}$ also holds for all $N\in\{1, \dots, 296\}$. 
		(The worst case is $N=49$ for which $g({X_0(N)})=1$.)
		\end{proof}

\begin{rema}\label{rema:genusbound}
Depending on the situation at hand, the bounds of Lemma \ref{lemm:genus.estimate} can be somewhat optimized.
For instance, one can deduce from $12 g({X_0(N)}) > \sqrt{N}(\sqrt{N}-{13})$ 
that $N\leq 13 g(X_0(N))$ for all $N\geq 28561$, and check with the help of a computer that, 
for all $N\geq 1$, 
\begin{equation}
N \leq \max\{3721, 13  g(X_0(N))\} 
\leq 3721 \max\left\{1 , 0.004 \,  g(X_0(N)) \right\}.
\end{equation}
The worst case happens for $N=3721$ when $g(X_0(N))=284$.

Starting from the inequality $12 (g({X_0(N)}) -1)\geq \sqrt{N}(\sqrt{N}-{13})$ proved above, 
one also deduces that $(\sqrt{N}-13/2)^2\leq (\sqrt{12g(X_0(N))} +11/2)^2$,
which leads to $N \leq 12 g(X_0(N)) + 2b\sqrt{12g(X_0(N))}+b^2$
with $b= 13/2 + 11/2= 12$. 
We obtain that 
\begin{equation}
 N \leq 12 g(X_0(N)) + 24\sqrt{12g(X_0(N))} + 144.
 \end{equation}
The last two displayed inequalities are more precise than Lemma \ref{lemm:genus.estimate} for larger genera.

For small genera, on the other hand, explicit computations using formula \eqref{eq:genus.modular} show that 
the modular curve $X_0(N)$ has genus zero if and only if 
$N\in\{1, 2, 3, 4, 5, 6, 7, 8, 9, 10, 12, 13, 16, 18, 25\}$. 
\end{rema}

\subsection{Bounding biseparable isogenies} 
 
Let $K$ be a function field with field of constants $k$, and fix a model $C/k$ of $K$ (see \S\ref{sec:intro.FF}). 
We write $g(C)$ for the genus of $C$ (which we also call the genus of $K$), and we let $G_K$ denote the absolute Galois group of $K$. 

Let $d\geq 1$ be an integer which is coprime to the characteristic of $K$. 
Assume that we are given a non-isotrivial elliptic curve $E$ over $K$, 
equipped with a cyclic, $G_K$-stable subgroup $H\subset E$ of order $d$.
Consider the modular curve $\Xcal_0(d)$, defined over $\Spec\Z[1/d]$, and 
write $X_0(d)_{/F}$ for its base change to a field $F$ whose characteristic is prime to $d$.
By the coarse moduli space interpretation of $\Ycal_0(d)$, 
there is a canonical map sending the $K$-isomorphism class of the pair $(E, H)$ to 
a $K$-rational point $P$ on $\Ycal_0(d)$.
From the given data $(E,H)$, we thus deduce a non-cuspidal $K$-rational point $P\in \Xcal_0(d)$. 

By definition, the point $P$ is given by a morphism $\Spec K \to \Xcal_0(d)$ over $\Spec\Z[1/d]$, 
which factors as a morphism $\Spec K \to X_0(d)_{/k}$ over $\Spec k$.
Since $\Spec K$ is the generic point of $C$, 
the latter induces a rational map $C\dashrightarrow X_0(d)_{/k}$. 
Since both $C$ and $X_0(d)_{/k}$ are smooth projective curves over $k$, 
this rational map extends to a morphism  $s_P : C\to X_0(d)_{/k}$ over $k$. 

Let $f_d:X_0(d)_{/k}\to X_0(1)_{/k} = \P^1_{/k}$ denote the degeneration morphism, 
which extends the map sending a pair $(E', H')$ to $j(E')$.
By construction, the morphism $j_E: C\to\P^1_{/k}$ deduced from the $j$-invariant of $E$ factors 
through $s_P$ and $f_d$. In other words, the diagram
\begin{equation}\label{eq:mod.curve.diag}
\begin{tikzcd}
 C 
 \ar[dr, "j_E"']
 \ar[rr,  "s_P"] && X_0(d)_{/k} 
 \ar[ld , "f_d"]\\
 & \P^1_{/k} &
\end{tikzcd}
\end{equation}
is commutative. We may now prove the following:
 
		\begin{prop}\label{prop:ff.bound}
		Let $E$ be a non-isotrivial elliptic curve over $K$.
		Let $H\subset E$ be a cyclic, $G_K$-stable subgroup of order $d$.
		If $d$ is coprime to the characteristic of $K$, we have
		\[ d \leq \cO{} \max\{1, g(C)\}.\]
		\end{prop}
		\begin{proof} 
		By the discussion above, the data $(E,H)$ of the proposition yields 
		a morphism $s_P : C\to X_0(d)_{/k}$ such that the diagram \eqref{eq:mod.curve.diag} commutes.
		The fact that $E$ is non-isotrivial implies that $j_E:C\to \P^1_{/k}$ is not constant which, in particular, 
		ensures that the morphism $s_P:C\to X_0(d)_{/k}$ is non-constant.

		A weak version of the Riemann--Hurwitz formula (see \cite[Theorem 7.16]{Rosen}) then entails that 
		the genus of~$X_0(d)_{/k}$ is no greater than that of $C$.
		Since $\Xcal_0(d)$ is a smooth curve over $\Spec\Z[1/d]$, 
		the genus of $X_0(d)_{/k}$ is equal to the genus of the complex Riemann surface $X_0(d)_{/\C}$.  
		Hence $g(C)\geq g({X_0(d)_{/\C}})$. 
		We then make use of the lower bound of Lemma \ref{lemm:genus.estimate}, 
		and conclude that  $d\leq \cO{} \max\{1, g({C})\}$. 
		\end{proof}

We now state and prove a variant of the previous proposition, 
which is significantly weaker (it is not uniform in the elliptic curve $E$) but which might prove more flexible.
This version may indeed be of interest for applications where one studies elliptic curves defined 
over a varying field $L$ whose degree over~$K$ is bounded 
(in which case a bound depending on the genus of $L$ may be too crude).
We provide an instance of such an application in  \S\ref{ss:back.isog.classes}.

		\begin{prop}\label{prop:ff.bound.V2}
		Let $K$ be a function field and let $L/K$ be a finite extension. 
		Let $E$ be a non-isotrivial elliptic curve over $L$,  
		and $H\subset E$ be a cyclic, $G_{L}$-stable subgroup of order $d$.
		Assuming that $d$ is coprime to the characteristic of $K$, we have
		\[d\leq   [L:K] \, \hmod(E).\]
		\end{prop}
		\begin{proof} 
		We fix a model $C'$ of $L$, and denote the constant field of $L$ by $k'$.
		Carrying out the argument preceding the previous proposition  with $L$ instead of $K$, 
		we deduce from the input $(E,H)$ an  $L$-rational point  $Q$ on the modular curve $\Ycal_0(d)$, 
		and a morphism $s_Q : C'\to X_0(d)_{/k'}$. 
		As in \eqref{eq:mod.curve.diag}, the $j$-invariant $j_E: C'\to\P^1_{/k'}$ factors through $s_Q$; 
		we thus have $j_E = f_d\circ s_Q$, where $f_d:X_0(d)_{/k'} \to \P^1_{/k'}$ is the degeneration map. 
		Since $j_E = f_d\circ s_Q$, it is clear that $\deg(j_E) \geq \deg(f_d)$.

		As was recalled in the previous subsection (see \eqref{eq:degree.degeneration}),  
		the degree of $f_d$ equals $[\mathrm{SL}_2(\Z):\Gamma_0(d)]=\psi(d)$.
		It is clear that $\psi(d)\geq d$, so that 
		\[\deg(j_E)\geq \deg(f_d) = {[\mathrm{SL}_2(\Z):\Gamma_0(d)]} = {\psi(d)} \geq {d}. \]
		Now, the degree of $j_E:C'\to\P^1_{/k'}$ is the degree of its divisor of poles $\div_\infty(j_E)\in\mathrm{Div}(C')$.
		As was remarked in \S\ref{ssec:hmod}, the modular height of $E$ satisfies $\hmod(E) = [L:K]^{-1} \, \deg(\div_\infty(j_E))$.
		Rearranging the terms in the inequality above then yields that
		$d \leq  [L:K] \, \hmod(E)$, 
		which concludes the proof.
		\end{proof}

\subsection{Proof of Theorem \ref{theo.small.isogenies}} 

We first prove a special case of Theorem \ref{theo.small.isogenies}, 
using the tools introduced in the previous two subsections:

		\begin{prop}\label{prop:small.ds.isogeny} 
		Let $E_1, E_2$ be two non-isotrivial elliptic curves over $K$.
		Assume that there exists a biseparable isogeny $\varphi : E_1 \to E_2$. 
		Then there exists a biseparable isogeny $\varphi_0 : E_1 \to E_2$ with
		\[\deg\varphi_0 \leq \cO{} \max\{1, g(K)\}. \]
		\end{prop}

If $K$ has characteristic $0$, the proof of this statement will conclude that of Theorem~\ref{theo.small.isogenies}.
The following proof uses arguments inspired by the ones of \cite[Lecture I, Proposition 7.1]{UlmerParkCity}, 
which proves a uniform upper bound for the order of the prime-to-$p$ $K$-rational torsion on an elliptic curve over~$K$.

		\begin{proof} 
		Fix a biseparable isogeny $\varphi:E_1\to E_2$, and consider the set of positive integers:
		\[\mathscr{D} := \left\{ \deg \phi, \ \phi: E_1\to E_2 \text{ biseparable isogeny}\right\} \subset \Z_{\geq 1}.\]
		By assumption, $\mathscr{D}$ contains a minimal element $d_0$, and there exists 
		a biseparable isogeny $\varphi_0 : E_1 \to E_2$ with $\deg\varphi_0 = d_0$. 
		Note that the degree $d_0$ of $\varphi_0$ is coprime to $p$, 
		for $\varphi_0$ is biseparable (see Lemma \ref{lemm:isog.ds.degree}).

		Let us now prove that $\varphi_0$ is cyclic \ie{}, that the group $H_0:=(\ker\varphi_0)(\bar{K})$ is cyclic.
		By the structure theorem for finite Abelian groups, and by the description of finite subgroups of $E_1$, 
		we know that $H_0$ is isomorphic to $\Z/m\Z\times\Z/n\Z$ for some integers $m,n\geq 1$ with $m\mid n$ and $mn = d_0$.
		In particular, the kernel of the multiplication-by-$m$ map $[m]:E_1\to E_1$ is a subgroup of $H_0$ 
		which is isomorphic to $(\Z/m\Z)^2$. Proposition \ref{prop:decomp.isog.factor} then implies that 
		the isogeny $\varphi_0$ can be factored as $\varphi_0 = \varphi_1\circ [m]$ for a unique isogeny $\varphi_1 : E_1\to E_2$.
		Taking degrees yields that  $\deg \varphi_0 = d_0 = \deg\varphi_1 \, m^2$.
		Being a divisor of $d_0$, $\deg\varphi_1$  must be coprime to $p$ so that, 
		by Lemma \ref{lemm:isog.ds.degree}, the isogeny $\varphi_1$ is biseparable.
		On the other hand, the degree of~$\varphi_0$ is minimal among all degrees of biseparable isogenies $E_1\to E_2$.
		We thus have $m=1$, and the kernel of $\varphi_0$ is isomorphic to $\Z/n\Z$.
		Hence $H_0$ is cyclic, as claimed.

		Since the degree of $\varphi_0$ is coprime to the characteristic of $K$, Lemma \ref{lemm:kernel.sep} shows that 
		the group $H_0$ is defined over a separable extension of $K$
		\ie{}, that the extension $K(H_0)/K$ is separable.
		Let $G_K$ denote the absolute Galois group of $K$. 
		We now prove that $H_0$ is stable under the action of $G_K$ on $E_1$ 
		(in other words, we show that $H_0$ is ``defined over $K$''). 

		To do so, we first prove that $^\sigma\varphi_0=\pm\varphi_0$  for all $\sigma\in G_K$.
		Let $\sigma\in G_K$ be an arbitrary automorphism. Since both $E_1$ and $E_2$ are defined over $K$,
		we get an isogeny $^\sigma\varphi_0: E_1\to  E_2$ of degree $d_0$.
		The composition of the dual $\widehat{\varphi_0}:E_2\to E_1$  with $^\sigma\varphi_0:E_1\to E_2$ yields 
		an endomorphism $\widehat{\varphi_0}\circ {^\sigma}\varphi_0$ of $E_1$.
		The curve $E_1$ being non-isotrivial, it has no non-trivial endomorphisms (see \S\ref{ssec:CM.isotriv}), 
		so that there exists an integer $n$ such that $\widehat{\varphi_0}\circ {^\sigma}\varphi_0 = [n]$ is the multiplication-by-$n$ map.
		Comparing degrees, we get that $n = \pm d_0$, whence 
		$\widehat{\varphi_0}\circ {^\sigma}\varphi_0 = [\pm d_0] = [\pm1]\circ \widehat{\varphi_0}\circ\varphi_0$.
		Therefore $\widehat{\varphi_0}\circ(^\sigma\varphi_0 - [\pm1]\circ\varphi_0) = 0_{E_1}$, and we deduce that 
		the image of the isogeny $^\sigma\varphi_0 - [\pm1]\circ\varphi_0$ is contained in the kernel of $\widehat{\varphi_0}$.
		Since the latter is finite, $^\sigma\varphi_0 - [\pm1]\circ\varphi_0$ must be constant equal to $0_{E_2}$. 
		Thus $^\sigma\varphi_0  = \pm\varphi_0$, as claimed.

		It then formally follows from the previous paragraph that $H_0$ is stable under the action of $G_K$.
		We may now apply Proposition \ref{prop:ff.bound} to the pair $(E_1, H_0)$, and we  infer that 
		$d_0 = \deg\varphi_0\leq \cO{} \max\{1,  g(K)\}.$
		\end{proof}

To conclude the proof of the general case of Theorem \ref{theo.small.isogenies}, we require the following lemma: 

		\begin{lemm}\label{lemm:reduction.4}
		Let $E_1, E_2$ be two non-isotrivial elliptic curves over $K$.
		We assume that $j(E_1)$ and $j(E_2)$ have the same inseparability degree, 
		and that there exists an isogeny $\varphi: E_1\to E_2$. 
		Then there exists a biseparable isogeny $\psi : E_1 \to E_2$ with $\deg\psi \mid \deg\varphi$.
		\end{lemm}
		\begin{proof} We write $\degins\varphi=p^a$ and $\degins\hat{\varphi}=p^b$. 
		By assumption we have ${\degins j(E_2)}={\degins j(E_1)}$, hence Corollary~\ref{coro:dualdeg} implies that 
		$p^a = \degins\varphi = \degins\hat{\varphi}= p^b$, so that $a=b$. 
		Using the decomposition of $\varphi$ provided by Proposition \ref{prop:decomposition.isogeny}, 
		one deduces that the integer $m := p^{-2a}\deg\varphi $ is coprime to $p$.

		By Proposition~\ref{prop:isogeny.reduction}, the isogeny $\varphi$ factors through the multiplication-by-$p^a$ map on $E_1$: 
		\ie{}, there exists an isogeny $\psi:E_1 \to E_2$ so that $\varphi = \psi \circ [p^a]$.
		By multiplicativity of degrees,  $\psi$  has degree $m$ dividing~$\deg\varphi$. 
		Since $m$ is coprime to $p$, Lemma \ref{lemm:isog.ds.degree} ensures that $\psi$ is biseparable.
		\end{proof}

		\begin{proof}[Proof of Theorem \ref{theo.small.isogenies}] 
		Let $E_1$ and $E_2$ be non-isotrivial isogenous elliptic curves over $K$. 
		If $K$ has characteristic $0$, Proposition \ref{prop:small.ds.isogeny} proves an assertion 
		which is equivalent to Theorem \ref{theo.small.isogenies}.
		We may therefore assume that the characteristic $p$ of $K$ is positive.

		Let $\varphi:E_1\to E_2$ be a $\bar{K}$-isogeny. As the sought statement is symmetric in $E_1$ and $E_2$ 
		we may, up to permuting the roles of $E_1$ and $E_2$ and using duals, assume that $\degins j(E_1)\geq\degins j(E_2)$.

		We then  let $f$ be the non-negative integer such that $\degins j(E_1)/\degins j(E_2) =p^f$ 
		and, to lighten notation, we write $E'_2$ for $E_2^{(p^f)}$. 

		The composition of $\varphi:E_1\to E_2$ with $\Frob{p^f} :E_2\to E'_2$ yields an isogeny $E_1\to E'_2$.
		By construction of~$f$ we have $\degins j(E_1) = \degins j(E'_2)$, so that
		Lemma \ref{lemm:reduction.4}  ensures the existence of a biseparable isogeny $\psi :E_1\to E'_2$. 
		By Proposition~\ref{prop:small.ds.isogeny}, there exists an isogeny $\psi_0 :E_1\to E'_2$ with 
		$\deg\psi_0 \leq \cO{}\max\{1, g(K)\}$. 

		Composing $\psi_0 :E_1\to E'_2$ with $\Vers{p^f}:E'_2\to E_2$, we obtain an isogeny $\varphi_0 : E_1 \to E_2$ of degree
		\[\deg\varphi_0 = p^f\deg\psi_0 \leq \cO{}\max\{1, g(K)\} \ \frac{\degins j(E_1)}{\degins j(E_2)},\]
		which concludes the proof of Theorem \ref{theo.small.isogenies}. 
		\end{proof}

\subsection{Number of $K$-isomorphism classes in a $K$-isogeny class} 

Let $K$ be a function field as above, and $E$ be a non-isotrivial elliptic curve over $K$.
For any $M\geq 1$, let us introduce the set
\[\Escr_K(E, M) := \left\{ E'/K : E' \text{ is $K$-isogenous to $E$, and }\degins j(E')\leq M\right\}\big/K\text{-isomorphism}.\]
(As before, if $K$ has characteristic $0$, one may ignore the condition on $\degins j(E')$.)

By construction, the elliptic curves $E'/K$ (whose $K$-isomorphism classes are) in $\Escr_K(E, M)$ have the same conductor as $E$.
In particular, they all have good reduction outside the set $S$ of places of $K$ where $E$ has bad reduction.
By a version of Shafarevich's theorem, there are finitely many 
$K$-isomorphism classes of elliptic curves over $K$ with good reduction outside $S$ and  
with bounded inseparability degree of $j$-invariant. The set $\Escr_K(E,M)$ is thus finite.
We prove an effective form of that statement 
(announced as Corollary~\ref{icoro:eff.sha} in the introduction), as follows:
		\begin{theo}\label{prop:effective.shaf}
		In the above setting, we have
		\begin{itemize}
		\item $\displaystyle \big|\Escr_K(E,M)\big| \leq 7^4 \max\{1, g(K)\}^2$ \ if $K$ has characteristic $0$,
		\item $\displaystyle \big|\Escr_K(E,M)\big| \leq 7^4 \max\{1, g(K)\}^2\left(\frac{\log M}{\log p} +1\right)$ \  
		if $K$ has characteristic $p>0$.
		\end{itemize}
		\end{theo}

We are grateful to one of the referees for remarks that led to an improvement on our original bound.
To prove the proposition, we take inspiration from section 2 of \cite{MasserW3}, 
which proves the analogue of Theorem~\ref{prop:effective.shaf} for elliptic curves over number fields.
We need the following two lemmas:

		\begin{lemm}\label{lemm:field.of.def} 
		Let $E_1, E_2$ be two non-isotrivial elliptic curves over $K$.
		We assume that there exists an isogeny $\varphi:E_1\to E_2$ which is defined over $K$.
		Then all isogenies $E_1\to E_2$ are defined over $K$.
		\end{lemm}
		\begin{proof}
		Since $E_1$ is non-isotrivial and since $E_1$, $E_2$ are isogenous, the $\Z$-module $\Hom(E_1, E_2)$ is free of rank~$1$; 
		we choose an isogeny $\varphi_0: E_1\to E_2$ such that $\Hom(E_1, E_2)=\Z\,\varphi_0$. 
		To conclude, it suffices to check that $\varphi_0$ is defined over $K$.
		
		We may write $\varphi$  as $\varphi = \varphi_0\circ [m]$ for some integer $m\neq 0$. 
		Let $\Hcal:=\ker\varphi$ denote the schematic kernel of~$\varphi$, and  
		$\Ecal_{1,m}:=\ker[m]$ denote the subgroup scheme of $E_1$ formed by $m$-torsion points.
		Both $\Hcal$ and~$\Ecal_{1,m}$ are (non-necessarily reduced) group schemes over $K$.
		Since $\varphi = \varphi_0\circ [m]$, $\Ecal_{1,m}$ is a subgroup scheme of~$\Hcal$. 
		
		We then use Theorem~1 in \cite[\S12]{MumfordAV} (which is a scheme-theoretic version of Proposition~\ref{prop:decomp.isog.factor}): 
		this shows that the isogeny~$\varphi$ factors uniquely (up to $K$-isomorphism) through $[m]$ 
		as~$\varphi = \psi\circ[m]$, where $\psi:E_1\to E_2$ is a $K$-isogeny. 
		The uniqueness of $\psi$ implies that $\varphi_0=\psi$, which is defined over $K$.
		\end{proof}

		\begin{lemm}\label{lemm:cyc.sbgps} 
		Let $E_1$ be a non-isotrivial elliptic curve over $K$. 
		For any real number $X\geq 1$, we denote by $\cycsbgp(E_1,X)$ 
		the number of cyclic subgroups $H$ of $E_1$ with $|H|\leq X$ and $|H|$ coprime to the characteristic of $K$. 
		Then we have $\cycsbgp(E_1, X)\leq X^2$.
		\end{lemm}		
		\begin{proof}
		Let $d\geq 1$ be an integer coprime to the characteristic $p$ of $K$.
		Given the structure of the $d$-torsion subgroup of $E_1$,
		 the number of cyclic subgroups of $E$ of order $d$ equals $\psi(d)$. 
		(Recall indeed that the group $(\Z/d\Z)^2$ contains $\psi(d)$ cyclic subgroups of order $d$.)
		For any $n\geq 1$, an elementary computation shows that $\psi(n) \leq \sigma(n)$, 
		where $\sigma(n)$ denotes the sum of divisors of $n$ (both functions are multiplicative, and the inequality holds for prime powers). 
		One easily sees that, for $X\geq 1$, 
		\begin{align*}
		\sum_{1\leq n\leq X} \sigma(n)
		&\leq \sum_{1\leq n\leq X} \sum_{d\mid n} d
		\leq  \sum_{1\leq d\leq X} \sum_{\substack{1\leq n\leq X \\ \text{s.t. }  nd \leq X}} d
		=  \sum_{  1\leq d \leq X}   d \, \left\lfloor\frac{X}{d}\right\rfloor \leq X^2.
		\end{align*}
		We therefore obtain the desired upper bound:  
		\[\cycsbgp(E_1, X) =\sum_{\substack{1\leq d\leq X \\ (d,p)=1}} \psi(d)
		\leq \sum_{1\leq n\leq X} \psi(n)  \leq \sum_{1\leq n \leq X} \sigma(n)\leq X^2.\]
		\end{proof}
		
		\begin{proof}[Proof of Theorem~\ref{prop:effective.shaf}] 
		For clarity, we first prove the assertion in the case when  $K$ has characteristic~$0$.
		Let $E'$ be an elliptic curve over $K$ whose $K$-isomorphism class lies in $\Escr_K(E, M)$, 
		and choose an isogeny ${\varphi_0 \in\Hom(E, E')}$ with minimal degree. 
		Proposition~\ref{prop:small.ds.isogeny} and its proof ensure that 
		$H_0 := (\ker\psi_0)(\bar{K})$ is cyclic of order
		$|H_0|=\deg\psi_0\leq c_K$, where we have set $c_K:=\cO{}\max\{1, g(K)\}$.
		The elliptic curve~$E'$ is then $K$-isomorphic to the quotient $E/H_0$, because 
		$\psi_0$ is defined over~$K$ by Lemma \ref{lemm:field.of.def}.
		This entails that the cardinality of $\Escr_K(E, M)$ is no greater than $\cycsbgp(E, c_K)$.
		With Lemma \ref{lemm:cyc.sbgps}, we conclude that $|\Escr_K(E,M)|\leq c_K^2$, 
		which proves the theorem in the characteristic $0$ case. 
		\\
		
		We now assume that the characteristic $p$ of $K$ is positive.
		For any non-negative integer $u$, we let
		\[ \Escr(u) := \left\{ E'/K : E' \text{ is $K$-isogenous to $E$, and }\degins j(E')=p^u\right\}\big/K\text{-isomorphism}.\]
		We can write the set $\Escr_K(E, M)$ as a disjoint union 
		$\Escr_K(E, M) = \bigcup_{0\leq u\leq m} \Escr(u)$ where  $m=\log M/\log p$.
		
		Let us prove that for any power $p^u$ of $p$, $|\Escr(u)|$ has cardinality at most $c_K^2$. 
		Given an integer $u\geq 0$ such that the set $\Escr(u)$ is non-empty, 
		let $E_1$ be an elliptic curve over $K$ whose $K$-isomorphism class lies in $\Escr(u)$.
		Any elliptic curve $E_2$ over $K$ whose isomorphism class lies in $\Escr(u)$ is then $K$-isogenous to $E_1$. 
		Furthermore, since the inseparability degrees of the $j$-invariants of $E_1$ and $E_2$ are the same, 
		we know from Lemma \ref{lemm:reduction.4} and Lemma \ref{lemm:field.of.def} that $E_1$ and $E_2$ are linked 
		by a biseparable isogeny which is defined over $K$. 
		We then carry out, as we may, the same argument as in the characteristic $0$ case above:
		this entails that the cardinality of $\Escr(u)$ is no greater than $\cycsbgp(E_1, c_K)$.
		Lemma~\ref{lemm:cyc.sbgps} then yields that $|\Escr(u)|\leq c_K^2$.
		
		If $p>0$, we therefore conclude that 
		\[ |\Escr_K(E, M)| \leq \sum_{0\leq u\leq m} |\Escr(u)| \leq (m+1)\, c_K^2 = c_K^2 \, \left(\frac{\log M}{\log p}+1\right),\]
		as was to be shown. 
		\end{proof}

\subsection{Back to isogeny classes}\label{ss:back.isog.classes} 

We close this paper by going back to the situation studied in \S\ref{ss:isog.classes}, 
and by recovering a finiteness result.
Let $K$ be a function field as above, and fix a non-isotrivial elliptic curve~$E$ over~$K$. 
Recall that $\Iscr_{\mathrm{bs}}(E)$ denotes the set of elliptic curves $E'/\bar{K}$ which are biseparably isogenous to~$E$. 
For any $B\geq 0$ and $D\geq 1$, consider the set 
\[ \Jscr'_K(E, B, D) 
= \big\{E'\in\Iscr_{\mathrm{bs}}(E) : \hmod(E')\leq B \text{ and }   [K(j(E')):K]\leq D \big\}\big/\,\bar{K}\text{-isomorphism}.\]
By Theorem~\ref{theo:isog.hmod}, all elliptic curves $E'\in\Iscr_{\mathrm{bs}}(E)$ satisfy $\hmod(E')=\hmod(E)$. 
The set $\Jscr'_K(E, B, D)$ is hence empty  if $B<\hmod(E)$ and, for $B\geq\hmod(E)$, we have 
\[\Jscr'_K(E, B, D) = \big\{E'\in\Iscr_{\mathrm{bs}}(E) : [K(j(E')):K]\leq D \big\}\big/\,\bar{K}\text{-isomorphism}.\]
Note that the set $\Jscr_K(E, B)$ studied in \S\ref{ss:isog.classes} is the union $\bigcup_{D\geq 1} \Jscr'_K(E, B, D)$.
We prove the following bound:

		\begin{prop}\label{prop:finite}	
		For any $B\geq 0$ and $D\geq 1$, the set $\Jscr'_K(E, B, D)$ is finite. 
		Moreover, we have
		\[|\Jscr'_K(E,B,D)| \leq  D^2 \,  \hmod(E)^2.\] 
		\end{prop}
		
		\begin{proof} 	
		Let $E'\in\Iscr_{\mathrm{bs}}(E)$ be an elliptic curve which is biseparably isogenous to $E$.
		Let $j(E')\in \bar{K}$ denote its $j$-invariant and $K':=K(j(E'))$. 
		If necessary, we replace $E'$ by a $\bar{K}$-isomorphic elliptic curve~$E_2$ defined over $K'$.
		We write $E_1$ for the base-change $E_1:=E\times_KK'$.
		
		By assumption, there exists a biseparable isogeny $\varphi:E_1\to E_2$.
		We may assume (see the proof of Proposition \ref{prop:small.ds.isogeny}) that 
		$\varphi = \varphi_0$ has minimal degree among all biseparable isogenies $E_1\to E_2$.
		By the same arguments as in the proof of Proposition \ref{prop:small.ds.isogeny}, 
		the group $H_0:=(\ker\varphi_0)(\bar{K})$ is then cyclic and stable under 
		the action of the absolute Galois group $G_{K'}=\Gal(\bar{K}/{K'})$.
		Proposition \ref{prop:ff.bound.V2} then yields the bound
		\[|H_0| = \deg\varphi_0 \leq [K':K] \, \hmod(E_1) \leq D \,  \hmod(E).\]
		The elliptic curve $E'\simeq E_2$ is then $\bar{K}$-isomorphic to the quotient $E_1/H_0$.
		Hence, the cardinality  of the set $\Jscr'_K(E,B,D)$ does not exceed the number $\cycsbgp(E_1, D \, \hmod(E))$, 
		introduced in the previous subsection, of cyclic subgroups $H$ of~$E_1$ 
		with order at most $D \, \hmod(E)$ and with $|H|$ coprime to the characteristic of $K$.
		By  Lemma~\ref{lemm:cyc.sbgps}, we thus have $|\Jscr'_K(E,B,D)| \leq (D \, \hmod(E))^2 = D^2 \,  \hmod(E)^2. $
		
		This proves the finiteness of $\Jscr'_K(E,B,D)$ and the  asserted upper bound on its cardinality.
		\end{proof}

\noindent\hfill\rule{7cm}{0.5pt}\hfill\phantom{.}

\paragraph{Acknowledgements}
RG is supported by the Swiss National Science Foundation, through the SNSF Professorship \#170565 awarded to Pierre Le Boudec, 
and received additional funding from ANR project ANR-17-CE40-0012 (FLAIR). 
He wishes to thank Universit\"at Basel for providing excellent working conditions. 
FP thanks the ANR project ANR-17-CE40-0012 (FLAIR) and the IRN GANDA (CNRS).

Both authors would like to thank 
Gabriel Dill, Marc Hindry, 
Pierre Le Boudec, Samuel Le Fourn, 
Harry Schmidt and Douglas Ulmer 
for stimulating conversations about the topic of this work. 
They also express their deepest gratitude to Ga\"el R\'emond for his valuable comments and suggestions on an earlier version of the article. 
Finally, the authors thank the referee for their careful reading and beneficial input. 


\newcommand{\mybref}[1]{\hspace*{-2pt}{\textcolor{orange}{\footnotesize $\uparrow$ #1}}}
\renewcommand*{\backref}[1]{\mybref{#1}}
\hypersetup{linkcolor=orange!80}

\small
\bibliographystyle{alpha}
\bibliography{Biblio_jinv.bib}


\normalsize\vfill
\noindent\rule{7cm}{0.5pt}

\smallskip
\noindent
{Richard Griffon} {(\it \href{richard.griffon@unibas.ch}{richard.griffon@unibas.ch})}  --
{\sc Departement Mathematik, Universit\"at Basel,} 
Spiegelgasse 1, 4051 Basel, Switzerland. 

\medskip
\noindent
{Fabien Pazuki} {(\it \href{fpazuki@math.ku.dk}{fpazuki@math.ku.dk})} --
{\sc Department of Mathematical Sciences, University of Copenhagen,}
Universitetsparken 5, 2100 Copenhagen \o{}, Denmark.

\end{document}